\documentclass[twoside,a4paper,11pt]{article}
\usepackage{amsfonts, amsbsy, amsmath, amsthm, amssymb, latexsym, verbatim, enumerate}

\usepackage{latexsym}
\newtheorem{propo}{Proposition}[section]
\newtheorem{defi}[propo]{Definition}
\newtheorem{lemma}[propo]{Lemma}
\newtheorem{corol}[propo]{Corollary}
\newtheorem{theor}[propo]{Theorem}
\newtheorem{theo}[propo]{Theorem}

\theoremstyle{definition}
\newtheorem{remar}[propo]{Remark}

\newcommand{\ld}{,\ldots ,}

\newcommand{\ra}{ \rightarrow }

\newcommand{\lan}{ \langle }
\newcommand{\ran}{ \rangle }

\newcommand{\diag}{\mathop{\rm diag}\nolimits}

\newcommand{\Id}{\mathop{\rm Id}\nolimits}

\newcommand{\Z}{\mathbb Z}

\newcommand{\ep}{\varepsilon}

\newcommand{\lam}{\lambda }

\newcommand{\up}{^{-1}}

\newcommand{\el}{\end{lemma}}

\newcommand{\om}{\omega }

\def\d12{{_{12}}}

\def\au{{automorphism }}

\def\ei{{eigenvalue }}
\def\eis{{eigenvalues }}

\def\f{{following }}

\def\ii{{if and only if }}

\def\ir{{irreducible }}

\def\irr{{irreducible representation }}

\def\mult{{multiplicity }}

\def\rep{{representation }}
\def\reps{{representations }}

\def\SL{{\rm SL}}

\begin{document}

\title{Subgroups of simple algebraic groups containing regular tori, 
and irreducible representations with \mult 1 non-zero weights}
\author{Donna Testerman and A.E. Zalesski}
\maketitle

\medskip
\begin{abstract}
Our main goal is to determine, under certain restrictions,  the maximal closed connected subgroups of simple algebraic groups
containing a regular torus. We call a torus 
\emph{regular} if its
centralizer is abelian.
 We also obtain some results of independent interest. In particular, we determine 
the \ir \reps of simple algebraic groups whose  non-zero weights 
occur with \mult 1.

\end{abstract}


\section{Introduction}\label{sec:intro}

Let $H$ be a simple algebraic group defined over an algebraically
closed field $F$ of characteristic $p\geq0$. 
The subgroups of $H$ containing  a maximal torus   are called subgroups of maximal rank. They play 
a substantial role in the structure theory of algebraic groups. 
A  well-known and important classical result (going back to Dynkin and Borel-De Siebenthal) states that
 every subgroup of maximal rank is either 
contained in a parabolic subgroup of $H$, or lies in the normalizer of a subsystem subgroups (See \cite[\S 13]{MTbook}.)
A subsystem subgroup $G$ is a semisimple subgroup of $H$, normalized by a maximal torus of $H$;
consequently the root system of $G$ is, in a natural way, a subset of the root system of $H$. Note that Seitz \cite{GS2}
obtained some analog of the Dynkin-Borel-De Siebenthal classification for finite Chevalley groups.  In addition,
there are many results in the literature concerning classifying subgroups of Chevalley groups 
over infinite fields which contain a maximal split torus; for a discussion and bibliography see \cite{Vav1,Vav2}.

Our aim here is to generalize the Dynkin-Borel-De Siebenthal result by replacing a maximal torus by a 
{\it regular torus}, that is, 
a torus whose centralizer in $H$ 
is a maximal torus of $H$. If $G$ is a closed subgroup of $H$ which contains a regular
 torus of $H$, then the maximal tori of $G$ 
are regular in $H$.
Hence, we would like to determine 
the closed connected subgroups $G$ of $H$ whose maximal tori are regular in $H$. 
In its most general form, the question as stated is intractable; one has simply to think about 
indecomposable representations of 
semisimple groups to see that even in the case where $H$ is of type $A_m$, a complete classification is not
 realizable at this time,
at least when ${\rm char}(F)>0$.
A tractable version of the above question, and the one which we consider here is the following:\medbreak

\noindent\emph{Problem} 1. Determine (up to conjugacy) the maximal closed connected subgroups $G$ of a simple algebraic 
group 
$H$ containing a  regular torus of $H$.
\medbreak

A reader familiar with the strategy for the study of subgroups of simple  algebraic groups may raise the natural question:
 to what extent does
the classification of the maximal closed connected subgroups of simple algebraic groups help in solving Problem 1. 
For example, can one simply ``inspect'' a list of maximal closed connected subgroups and arrive at the answer to 
the problem? We claim that this is not possible. Indeed,
consider for example the case where $H$ is a classical group with natural module $V$. Here the classification of 
maximal closed connected subgroups
 reduces to a question about the maximality
of simple subgroups $G\subset H$ acting irreducibly and tensor-indecomposably on $V$. Since the overwhelming majority
of irreducible tensor-indecomposable subgroups of $H$ are maximal among closed connected subgroups of $H$, the 
classification 
consists of 
listing explicitly the non-maximal such subgroups. (See \cite[Thm. 3]{Seitzclass}.) Therefore, it is impossible to decide 
by simple examination of maximal subgroups 
$G$ of $H$ which of them contains a regular torus of $H$, at least for classical groups $H$. The situation for the exceptional
algebraic groups $H$ is somewhat different and will be discussed later. 

The solution to Problem 1 must include of course all subgroups of $H$ of maximal rank.
As mentioned above, such subgroups are either non reductive and therefore equal to a maximal parabolic subgroup (by \cite{BT}), or are 
described by
the Dynkin-Borel-De Siebenthal classification. Therefore, we will
henceforth consider reductive subgroups $G$ with ${\rm rank}(G) < {\rm rank}(H)$.

We use different approaches to deal with Problem 1 depending on whether $H$ is of classical or of exceptional type. 
In the former case
our strategy is to reduce Problem 1 to 
the recognition of linear representations of simple algebraic groups $G$ whose weights satisfy certain specified 
properties.  
Denote by $V$ the natural module for $H$. Suppose first that $H$ is of type $A_m$.  
Then the embedding $G\ra H$ can be viewed as a representation, and the condition for a
torus $T$ of $G$ to be a regular torus in $H$ 
can be expressed in geometric terms. Specifically, using the language of representation theory, 
$T$ is a regular torus in $H$ \ii all $T$-weight spaces of $V$
are one-dimensional. Therefore, one needs to determine the \reps of simple algebraic groups 
all of whose weight spaces are one-dimensional.
In this form, the problem we are discussing
was considered in a fundamental manuscript of Seitz \cite{Seitzclass}, which was the first major step toward extending Dynkin's
 classification of the maximal closed connected subgroups of the simple algebraic groups defined over ${\mathbb C}$ to the 
groups defined over fields of positive characteristic.  For his purposes, Seitz only needed infinitesimally irreducible
representations satisfying the condition on weight spaces; his result was later extended to general irreducible representations
in  \cite{SZ1}. (We quote the result in Proposition \ref{sz7} below.)

Now suppose that $H$ is a classical group not of type $A_m$; again we view the embedding $G\to H$ as a representation of $G$.
As $G$ is assumed to be maximal among closed connected subgroups, and not containing a maximal torus of $H$, 
a direct application of \cite[Thm. 3]{Seitzclass} shows that either $G$ acts irreducibly on $V$ or
the pair $(G,H)$ is $(B_{m-k}B_k,D_{m+1})$, for some $0\leq k\leq m$. In the latter case, it is
 straightforward to show that $G$ contains a regular torus of $H$. See Section~\ref{sec:red} for details.

Turning now to the \ir subgroups $G\subset H$ containing a regular torus of $H$, we see that it is easier 
to determine such subgroups when $H$ is of type 
$B_m$ or $C_m$ due to the fact that 
a regular torus of $H$ is regular in ${\rm SL}(V)$. Hence, we may refer to the previously mentioned
classification of representations having $1$-dimensional weight spaces.
If $H=D_m$, we need the following more precise result, which we prove in Section~\ref{sec:sscl}. Recall that the \emph{multiplicity} of a weight in a representation is the
dimension of the corresponding weight space.

\begin{propo}\label{55yy} Let $H$ be a classical type simple algebraic group over $F$. 
Let
$V$ be the natural $FH$-module and
let $T$  be a (not necessarily maximal) torus in $H$. 
Then the \f statements are equivalent:\begin{enumerate}[]

\item{\rm (1)} $T$ is a regular torus in $H$;
\item{\rm (2)} Either all $T$-weights  on V have multiplicity $1$,  
or $H$  is of type $D_m$,
and exactly one $T$-weight has \mult $2$ and all other $T$-weights have multiplicity $1$.
\end{enumerate}
\end{propo}

The  above proposition motivates our consideration of \ir \reps of
simple algebraic groups $G$ having at most one $T$-weight space of dimension greater than $1$ 
(where $T$ is now taken to be a maximal torus of $G$). This is another main theme of this paper which contributes to the 
study of the pattern of weight multiplicities of irreducible representations of simple algebraic groups. 
This topic has been studied from various points of view, see for example \cite{Seitzclass, BOS1, SZ, Premet88}. 
In \cite[Propositions 3.4 and 3.6]{BenkartLemire}
 Benkart, Britten and Lemire determined infinite-dimensional modules for the complex simple Lie algebras
having a weight space decomposition with all weight spaces of dimension 1. Their  results hint at some possible 
connection between 
weight multiplicities in irreducible representations of 
 semisimple groups defined over a field of positive characteristic and weight multiplicities in 
appropriate infinite-dimensional modules for semisimple complex Lie algebras.

Before stating our main result in this direction, we introduce the following notation. For  a  dominant weight  
$\lambda\in X(T)$, the group of rational characters of $T$, we denote by $L_G(\lam)$ the irreducible
$FG$-module of highest weight $\lambda$. 

\begin{theor}\label{mt11} Let $G$ be a simple algebraic group over $F$ and let $T\subset G$ be a maximal torus.
\begin{enumerate}[]
\item{\rm (1)} Let $\lam\in X(T)$ be a non-zero dominant weight. If ${\rm char}(F)=p>0$, assume in addition that 
$\lambda$ is $p$-restricted.
All  non-zero weights
of $L_G(\lam)$ are of \mult $1$ if and only if $\lam$ is as in Table $1$ or Table $2$.
\item{\rm (2)} Let $\mu\in X(T)$ be a non-zero dominant weight such that all non-zero 
weights of $L_G(\mu)$ are of \mult $1$. Then either all weights of $L_G(\lambda)$ have multiplicity $1$, or
$\mu=p^k\lam$ for some integer $k\geq 0$ and $\lam$ as in  Table $1$ or Table $2$, where $k=0$ if 
${\rm char}(F)=0$.
\end{enumerate}
\end{theor}

In order to interpret the above-mentioned tables, we require the following definition.

\begin{defi} Let $G$ be a semisimple algebraic group with maximal torus $T$. 
We denote by
\emph{$\Omega_2(G)$} the set of $p$-restricted weights $\lam\in X(T)$ such that
all non-zero weights of $L_G(\lam)$ have \mult $1$, 
and by
\emph{$\Omega_1(G)$} the set of weights $\lambda\in\Omega_2(G)$ such that
all  weights of $L_G(\lam)$ have \mult $1$.

\end{defi}

It is perhaps worth pointing out that there are no infinite families of examples in Table 2, unlike the situation with
modules all of whose weight multiplicities are 1. Moreover, the set of weights $\Omega_2(G)\setminus\Omega_1(G)$
is independent of the characteristic; hence one has the same list in characteristic 0 and in positive characterisic $p$.
Note that Table 2 also contains the data on the multiplicity of the zero weight in $L_G(\lam)$.
We think that this result may be useful for other applications.

In applying the above result to obtain a solution to Problem 1 for a classical type group $H$,
we observe that certain irreducible representations $\rho_1$, $\rho_2$ of a simple group $G$,
with highest weights $\lambda_1$, respectively $\lambda_2$,
give rise to the same subgroup of $H$; that is, we find that 
$\rho_1(G)=\rho_2(G)$. For example, this happens for $\lambda_2 = p^k\lambda_1$, 
or if $\rho_1$ is obtained from $\rho_2$ by applying a graph
automorphism of $G$. In particular, this occurs for certain pairs $\lambda_1,\lambda_2$ such that 
all non-zero weights of $L_G(\lambda_1)$ and $L_G(\lambda_2)$ occur with multiplicity at most one. 
Such coincidences are clear and so we will not list them explicitly.

We now consider the image of the simple group $G$ under an irreducible representation $\rho$ having highest weight
as in Theorem~\ref{mt11}.
To decide whether $\rho(G)$ contains a regular torus of the corresponding classical group, we must determine
whether $\rho(G)$ stabilizes a non-degenerate quadratic or alternating form on the associated module, and apply
Proposition~\ref{55yy}. This will be carried out in Section~\ref{sec:forms} for tensor-indecomposable representations. 
See Section~\ref{sec:red}
for a discussion of the general case.

Now we turn to the case where  $H$ is of exceptional Lie type. In contrast with the classical group case, 
the classification of maximal positive-dimensional closed subgroups of $H$ is explicit \cite[Table 1]{LS2};
we analyse the  maximal connected subgroups which are not  of maximal rank, 
and decide for which of
 these a maximal torus is regular in  $H$. This is done in Proposition~\ref{dd1}. Here our method uses 
a different aspect of \rep theory than that used in the classical group case.
It is based upon the following fact.

\begin{propo}\label{lie} Let $H$ be a connected algebraic group over $F$, with Lie algebra ${\rm Lie}(H)$. Let $T\subset H$
be a (not necessarily maximal) torus. 
The torus $T$ is regular in $H$ \ii $\dim C_{{\rm Lie}(H)}(T) = {\rm rank}(H)$. 
\end{propo}

This follows from the fact that ${\rm Lie}(C_H(S)) = C_{{\rm Lie}(H)}(S)$, for any torus $S\subset H$.
 (See \cite[Prop. A. 18.4]{Hu}). Now 
$T$ lies in a maximal torus $T_H$ of $H$ 
and so ${\rm Lie}(T_H)=C_{{\rm Lie}(H)}(T_H)\subset C_{{\rm Lie}(H)}(T)$.  
Hence, $T$ is regular in $H$ if and only if $C_{{\rm Lie}(H)}(T) = {\rm Lie}(T_H)$. If $G$ is
 explicitly given as a subgroup of $H$, one 
can determine 
the composition factors of the restriction of ${\rm Lie}(H)$ as $FG$-module; indeed this information is available in \cite{LS2}.
 Next, for every composition factor, we
determine the \mult of the zero weight.
 Then $C_{{\rm Lie}(H)}(T)={\rm Lie}(T_H)$ \ii the sum 
of these multiplicities equals the rank of $H$. Our solution to Problem 1 is explicit and the result is given in Table~\ref{tab:exc}. 

In contrast, for classical groups the main ingredient is the \f statement:

\begin{theo}\label{mt22} Let $G$ be a simple algebraic group over $F$
and $\rho:G\to {\rm GL}(V)$ an irreducible rational representation with $p$-restricted highest weight $\lambda$.
Let $H\subset {\rm GL}(V)$ be the smallest simple classical group containing $\rho(G)$ and assume 
$\rho(G)$ contains a regular torus of
$H$. Then the pairs $(G,\lambda)$ are those appearing in Tables{\rm ~\ref{tab:Omega1-symp},~\ref{tab:omega2-even-orth},~\ref{tab:Omega1-odd-orth},~\ref{tab:orthp=2},~\ref{tab:nonselfdual},~\ref{tab:p=2symp}}.
\end{theo}

Observe that Theorem~\ref{mt22} includes the case of ${\rm char}(F)=0$. Moreover, we do not 
require $\rho(G)$ to be maximal in $H$. The specific case of maximal subgroups of classical
type groups is treated in Theorems~\ref{ma1} and \ref{ma2}.  

The main result for subgroups of exceptional type 
groups $H$ does, however, require
the hypothesis of maximality.

\begin{theo} Let $H$ be an exceptional simple algebraic group over $F$ and $M\subset H$ a maximal closed
connected subgroup containing a regular torus of  $H$.
Then either $M$ contains a maximal torus of $H$ or the pair 
$(M,H)$ is as in in Table~$\ref{tab:exc}$.  \end{theo}

Next we outline some possible applications of the above results.  

We expect to use our results  
for recognition of 
linear groups, and more generally, for recognition of subgroups of algebraic groups
that contain an element of a specific nature. This is the principal motivation for our 
consideration of Problem 2 below.  A
semisimple element of $H$ is regular if it centralizes no nontrivial unipotent element. (See Definition~\ref{reg}
and the discussion following the definition.)
We now apply the following standard result.

\begin{lemma}\label{ca1} {\rm \cite[Prop. 16.4]{Hu}}
Let $K$ be a connected algebraic group. A torus $S$ of $K$  is
regular \ii $S$ contains a regular element.
\end{lemma}

As every semisimple element of $G$ belongs to a torus in $G$, 
Problem 1 is therefore equivalent to  Problem 2.
\medskip

\noindent {\emph{Problem} $2$. Let $H$ be a simple algebraic group over $F$. Determine (up to conjugacy) all maximal
closed connected subgroups $G$ of $H$ such that $G$ contains a regular semisimple element of $H$.

\medskip
In order to place problem 2 in a more general context, let us recall 
a few  well-known ``recognition'' results. Historically, Mitchell's classification of finite
\ir  primitive linear groups over the complex numbers generated by reflections (1914) was among the most striking. 
Much more recently, Wagner \cite{Wa1,Wa2-3} and Zalesski\u{i}-Sere\v{z}kin \cite{ZS2} classified 
\ir linear groups 
over finite fields generated by reflections. In \cite{He1}, Hering determined the subgroups of $GL(n,q)$ 
generated by irreducible elements of prime order.  

Another potential application of our result is to the study of 
maximal subgroups of
finite simple groups of Lie type, containing a regular element (of the ambient algebraic group). 
This problem was explicity stated by Walter in \cite{Walter}.
A related problem is studied in \cite{SZ_Comm1998, SZ_Comm2000}.
There the authors analyze the irreducible $p$-modular
representations of finite
Chevalley groups
$\rho:G\to
H=GL(V)$ whose image contains a regular semisimple element of $H$,
in the case where  $p={\rm char}(F)$ is
 the defining characteristic of $G$. 

There is a series of results on recognition of finite linear  groups containing a unipotent element of 
particular shape.  
Pollatsek \cite{Po} in characteristic 2, and Wagner \cite{Wa}
 and Zalesski\u{i}-Sere\v{z}kin \cite{ZS1} for odd characteristic, obtained a classification of \ir linear 
groups over finite fields generated by
transvections. Next Kantor \cite{Ka} 
classified the subgroups  of finite classical groups $H$ generated by 
long root elements of  $H$. In \cite{Co1, Co2}, Cooperstein considered the case of
finite groups of exceptional Lie type.  
In fact, Wagner also treated indecomposable subgroups, whereas 
Kantor and Cooperstein assumed that the subgroup under consideration  has
 no normal unipotent subgroup. 

Thompson \cite{Th} initiated the study of finite irreducible linear groups generated by elements $x$ with $(x-\Id)^2=0$  
(such elements are often called ``quadratic'').
Irreducible representations of simple algebraic groups whose image contains
a  quadratic unipotent element were determined by Premet and Suprunenko
in \cite{PS}. Closed simple subgroups of simple algebraic groups containing long root elements
were identified in Liebeck and Seitz \cite{LS1}, 
where they also consider quasisimple subgroups of the finite groups arising as fixed points of some 
endomorphism of the algebraic group.

A result of a similar nature would be that of classifying linear groups $G$ over a 
finite field generated by 
 regular unipotent elements, in place of  root elements as in the above publications. 
(See Definition~\ref{reg} below.) The case where $G$
 is a connected algebraic group 
 is a generic version of this problem.  
 This  has been considered by  Suprunenko in \cite{Su}. The case where $H$
 is an arbitary simple algebraic group (in place of a special linear group)  
was considered by Saxl and Seitz \cite{SS} and 
the present authors \cite{TeZ}.  
In \cite{Su}, $G$ is assumed to be an irreducible subgroup of the special linear group, 
 in \cite{SS} $G$ is assumed to be maximal among positive-dimensional closed subgroups of 
an arbitrary simple algebraic group $H$,
 whereas in \cite{TeZ}, we assume that
$G$ is connected reductive.
  Note that restrictions of this kind are unavoidable, as otherwise 
an explicit classification is impossible.  The results in \cite{SS} and \cite{TeZ} 
can be viewed as a kind of mirror image of those 
by Kantor and Cooperstein, in \cite{Ka, Co1},

A general principle unifying the above results is to recognize a subgroup of a simple algebraic group 
by a conveniently  verifiable property of a single element of the subgroup. That is, given a single element
$g$ of $H$ described in convenient terms,
 determine the closed subgroups $G$ of $H$ containing $g$.

Following the above logic, we are concerned here with classifying the connected reductive algebraic 
subgroups of $H$ generated by regular elements, with further potential applications to finite algebraic groups. 
 We recall the following

\begin{defi}\label{reg} An element $x$ of a connected reductive algebraic group $H$ is said to be 
\emph{regular} if $\dim C_H(x) = {\rm rank}(H)$.
\end{defi}

Note that for a semisimple element $x\in H$, this is equivalent to saying that 
$C_H(x)^\circ$ is a torus, 
or equivalently  that $C_H(x)$ contains no 
non-identity unipotent element. (See \cite[Cor. 4.4]{SpSt} and  \cite[14.7]{MTbook}.) This is also 
equivalent to saying that  $C_H(x)$ has an abelian normal subgroup of finite index. 

In contrast with regular unipotent elements, the set of regular semisimple elements in $H$
 contains infinitely
many $H$-conjugacy classes. This suggests that the classification problem for subgroups 
containing such elements
lies on a much higher level of difficulty. Nonetheless, our results lead to some progress 
in this direction. 

It is of course natural to consider Problem 2 for arbitrary regular elements (i.e., 
not necessarily semisimple), 
especially
because as mentioned above, the case where $g$ is a regular unipotent element was 
already studied and solved in \cite{SS}, \cite{TeZ}, and \cite{Su}.
The following result is proven in Section~\ref{sec:nonss}.

\begin{theo}\label{mt33}
Let $H$ be a connected reductive group over $F$. 
Let $G\subset H$ be a closed connected reductive subgroup. If $G$ contains a regular element of $H$ 
then $G$ contains a regular semisimple element of $H$.
\end{theo}

The resolution of Problem 2 for arbitrary regular elements is reduced, by Theorem~\ref{mt33}, to inspecting the subgroups 
$G$ of $H$ that contain a  
regular semisimple element of $H$.  As mentioned above, the case of
regular unipotent elements has been studied and resolved in \cite{SS,TeZ}, while in the current manuscript,
we treat the case of regular semisimple elements. 
To single out these cases is rather natural, but we do not think it 
particularly informative
to analyse in detail other cases where $g$ is regular but neither semisimple nor unipotent. 
We indicate after the proof of   
Theorem~\ref{mt33}  how one can handle the general case in terms of knowledge of
the cases where $g$ is either semisimple or unipotent. 

\medskip

\noindent {\bf{Notation.}} We fix here the notation and terminology to be used throughout
 the paper. We write ${\mathbb N}$ for the set of non-negative integers.
Let $F$ be an algebraically closed field, of characteristic $0$ or of
prime characteristic ${\rm char}(F)=p>0$.
For a natural number $a\geq 1$, we write $p\geq a$ (respectively $p>a$) to mean that either ${\rm char}(F) = 0$ or 
${\rm char}(F) = p\geq a$ (resp.  ${\rm char}(F) = p> a$). For a linear algebraic group $X$ defined over $F$, we write $X^\circ$ for the connected
component of the identity. All groups considered will be linear algebraic groups over $F$, and we will not repeat
this each time.

Let $G$ be a reductive algebraic group over $F$. All $FG$-modules are assumed to be rational, and we will not make 
further reference to this fact.
We fix a maximal torus and Borel subgroup $T\subset B$ of $G$, the root system $\Phi(G)$ with respect to $T$,   
a set of simple roots $\{\alpha_1,\dots,\alpha_n\}$ corresponding to $B$,  and the corresponding
fundamental dominant weights $\{\omega_1,\dots,\omega_n\}$. Write $X(T)$ for the group
of rational characters on $T$. Given a dominant weight $\lam\in X(T)$, we
 write $L_G(\lam)$ for the irreducible $FG$-module with highest weight $\lam$.  A dominant weight $\lambda\in X(T)$ is said to be \emph{$p$-restricted} if either 
${\rm char}(F)=0$ or ${\rm char}(F)=p$ and $\lambda=\sum a_i\omega_i$ with $a_i<p$ for all $i$.
Recall that a weight $\mu\in X(T)$ is said to be \emph{subdominant} 
to $\lambda$ if $\lambda$ and $\mu$ are dominant weights and $\mu =\lambda-\sum a_i\alpha_i$ for some $a_i\in{\mathbb N}$.
For an $FG$-module $M$ and a weight $\mu\in X(T)$, we let $M_\mu$ denote the $T$-weight space 
corresponding to the weight $\mu$.
Set $W_G: = N_G(T)/T$, the Weyl group of $G$ and write $s_i$ for the reflection in $W_G$
corresponding to the simple root $\alpha_i$. We label Dynkin diagrams as in Bourbaki \cite{Bourb4-6}. 

We will assume that $p>2$ when $G$ is simple of type $B_n$, since for our purposes,
 when $p=2$ the group $B_n$ can be treated as a group of type $C_n$. 
When $G$ is a classical type simple algebraic group, by the ``natural''
module for $G$ we mean $L_G(\omega_1)$. If $G$ is simple of type $B_n$, $C_n$, or $D_n$,
it is well-known that $G$ preserves a non-degenerate bilinear or quadratic form on the natural module.


\begin{table}
$$\begin{array}{|l|c|}\hline
~~~~~~~~~~G& \Omega_1(G)\setminus \{0\} \\
\hline
\hline
A_1&a\om_1, 1\leq a<p\\
\hline
A_n, n>1&a\om_1, b\om_n, \ 1\leq a,b<p\\
& \om_i,\  1<i< n\\
&c\om_i+(p-1-c)\om_{i+1}, \  1\leq i<n,\  0\leq c<p \\
\hline
B_n, n>2, p>2 &\om_1,\ \om_n \\
\hline
C_n, n>1, p=2&\om_1,\ \om_n\\
\hline 
C_2,p>2&\om_1,\ \om_2,\ 
\om_1+\frac{p-3}{2}\om_2,\ \frac{p-1}{2}\om_2\\
 \hline
C_3&\omega_3\\
\hline
C_n,n>2,p>2&\om_1,\ \om_{n-1}+\frac{p-3}{2}\om_n,\ \frac{p-1}{2}\om_n\\
 \hline
D_n, n>3&\om_1,\ \om_{n-1},\  \om_n\\
 \hline
E_6&\om_1, \ \om_6\\
\hline
E_7&\om_7  \\
\hline
F_4, p=3&\om_4 \\
\hline
G_2, p\neq 3&\om_1\\
\hline 
G_2, p= 3&\om_1,\ \om_2\\
\hline
\end{array}$$
\caption{Irreducible $p$-restricted $G$-modules with all weights of
multiplicity 1}\label{tab:omega1}
\end{table}

\begin{table}
$$\begin{array}{|l|l|c|c|}
\hline
G& \mbox {conditions}&\Omega_2(G)\setminus \Omega_1(G) &\mbox{weight } 0 \mbox{ \mult}\\
\hline
A_n,&n>1,(n,p)\neq (2,3)&\om_1+\om_n& \begin{cases} n&\mbox{ if }p\!\not|(n+1)\\ n-1&\mbox{ if } p|(n+1)\end{cases}\\
&&& \\
A_3& p>3&2\om_2& 2     \\

 \hline
B_n &n>2,p>2& \om_2&  n \\
&&&\\
&-&2\om_1&\begin{cases}n&\mbox{ if } p|(2n+1)\\  n+1&\mbox{ if } p\!\not|(2n+1)\end{cases}   \\
\hline
C_n&n>1&2\om_1 &n\\
& n>2, (n,p)\ne (3,3)&\om_2& \begin{cases}n-2& \mbox{ if } p|n\\  n-1& \mbox{ if } p\!\not|n\end{cases}\\
C_2&p\ne5&2\om_2&2\\
C_4&p\ne 2,3&\omega_4&2\\
\hline
D_n&n>3,~p>2&2\om_1&  \begin{cases}n-2 &\mbox{ if } p|n\\  n-1& \mbox{ if }
p\!\not|\,n\end{cases}\\
 && \om_2& n\\
 D_n&n>3, p=2& \om_2&n-\gcd(2,n) \\
\hline
E_6&&\om_2&\begin{cases}5 &\mbox{ if } p=3\\  6& \mbox{ if } p\neq 3\end{cases} \\
\hline 
E_7&&\om_1 &\begin{cases}7 &\mbox{ if } p\ne2\\  6& \mbox{ if } p=2\end{cases}\\
 \hline
E_8&&\om_8 & 8\\
\hline
F_4&&\om_1 &\begin{cases}2& \mbox{ if } p=2\\  4 &\mbox{ if } p\ne2\end{cases}   \\
 & p\neq 3&\om_4 &2 \\
\hline
G_2& p\neq 3&\om_2&2\\
\hline
\end{array}$$
\caption{Irreducible $p$-restricted $G$-modules with non-zero
weights of multiplicity 1 and whose zero weight has \mult 
greater than $1$.}\label{tab:omega2}
\end{table}

\section{Initial reductions}\label{sec:red}
Let $H$ be a simple algebraic group over $F$ and let $G$ be a maximal closed 
connected subgroup of $H$ which contains 
a regular torus of $H$. If $G$ is not reductive, then \cite{BT} implies that 
$G$ is a maximal parabolic subgroup of $H$ and 
hence contains a maximal torus of $H$. Henceforth we will restrict our attention to 
reductive subgroups $G$.

Consider now the case where $H$ is a classical type simple algebraic group
with natural module $V$. 
We use the general reduction theorem,  \cite[Thm. 3]{Seitzclass}, on maximal closed connected subgroups of $H$;
this allows us to restrict our considerations to irreducibly acting subgroups of $H$.
For a detailed
discussion of this see \cite[\S18]{MTbook}.

Let $G$ be maximal among closed connected subgroups. Then one of the following holds:
\begin{enumerate}
\item $G$ contains a maximal torus of $H$.
\item $H$ is of type $D_m$, $V = U\oplus U^\perp$, $U$ is an
odd-dimensional non-degenerate subspace with respect to the bilinear form on $V$, $2\leq\dim U\leq \dim V-2$, and 
$G = {\rm Stab}_H(U)^\circ$.
\item $V=V_1\otimes V_2$, each of $V_1$ and $V_2$ is equipped with either the zero form (in case $V$ has no non-degenerate 
$H$-invariant form),
or a non-degenerate bilinear or quadratic form, and the form on $V$ is obtained as the product form. Moreover,
 $G$ is the connected component of $({\rm Isom}(V_1) \otimes{\rm Isom}(V_2))\cap H$. Note that if $V$ is equipped with a 
quadratic form and ${\rm char}(F)=2$, then $G={\rm Sp}(V_1)\otimes {\rm Sp}(V_2)$. 
\item $G$ is simple
acting irreducibly and tensor indecomposably on $V$.
\end{enumerate}

In the first three cases, it is straightforward to show that $G$ contains a regular torus of $H$. Hence we are reduced to
 considering simple subgroups which act irreducibly and tensor indecomposably on $V$.


\section{Irreducible representations of simple algebraic\\ groups 
whose  non-zero weights are of \mult 1}\label{sec:omega2}

In this section, we obtain the classification of irreducible representations
all of whose non-zero weights have multiplicity 1, thereby establishing Theorem~\ref{mt11}.
Throughout this section we take $G$ to be a simply connected simple algebraic group over $F$.
The rest of the notation will be as fixed in Section~\ref{sec:intro}. As discussed in Section~\ref{sec:intro}, 
we will require the following classification.

\begin{propo}\label{sz7}
Let $\lambda\in X(T)$ be a non-zero dominant weight.
\begin{enumerate}[]
\item{\rm (1)} Assume in addition that $\lambda$ is $p$-restricted. 
 Then all weights of
$L_G(\lam)$ are of \mult $1$ \ii $\lam$ is as
in the second column of Table~$\ref{tab:omega1}$.  In other words,  the set $\Omega_1(G)\setminus \{0\}$
is as given in Table~$\ref{tab:omega1}$. 
\item{\rm (2)} Suppose that $p>0$ and $\lam$ is not $p$-restricted, so
 $\lam=\sum_{i=0}^k
p^i\lam_i$, where $\lam_i$ are $p$-restricted and $\lambda_i\ne 0$ for some $i>0$. Then
all weights of $L_G(\lam)$ are of \mult $1$ if and only if  the following hold:
\begin{enumerate}[] 
\item{\rm (a)} for all $0\leq l\leq k$, $\lam_l=0$ or
$\lam_l\in\Omega_1 (G)$,   and 
\item{\rm (b)} for all $0\leq l< k$, 

\quad $(\lam_l,\lam_{l+1})\neq (\om_n,\om_{1})$ if $p=2$,
$G=C_n$;

\quad $(\lam_l,\lam_{l+1})\neq (\om_1,\om_{1})$ if $p=2$,
$G=G_2$;

\quad $(\lam_l,\lam_{l+1})\neq (\om_2,\om_{1})$ if $p=3$,
$G=G_2$.
\end{enumerate}
\end{enumerate}
\end{propo}

\begin{proof} The result follows from {\rm \cite[6.1]{Seitzclass}} and \cite[Prop.2]{SZ1}.
\end{proof}

We now collect some results on dimensions of certain
weight spaces in infinitesimally irreducible $FG$-modules. The proofs can be found in 
\cite{Seitzclass}, \cite{Testexc} and
\cite{BGT}.

\begin{lemma}\label{wtlem1} Let $\lambda\in X(T)$ be a $p$-restricted
dominant weight.\begin{enumerate}[]
\item{\rm (1)} If $G$ is of type $A_n$ and
$\lambda=a\om_j+b\om_k$, with $ab\ne 0$, then the
multiplicity of the weight
$\lambda-\alpha_j-\alpha_{j+1}-\cdots-\alpha_k$ in $L_G(\lam)$
is $k-j+1$ unless $p\vert (a+b+k-j)$, in which case the
multiplicity is $k-j$. 

\item{\rm (2)} If $G$ is of type $A_n$ and $\lambda=c\om_i$
for some $1<i<n$ and $c>1$, then the multiplicity of the weight
$\lambda-\alpha_{i-1}-2\alpha_i-\alpha_{i+1}$ in $L_G(\lam)$ is
$2$ unless $c=p-1$ in which case the multiplicity is $1$.

\item{\rm (3)} If $G$ is of type $B_2$ and $\lambda=a\om_1+b\om_2$,
with $ab\ne 0$, the multiplicity of the weight
$\lambda-\alpha_1-\alpha_2$ in $L_G(\lam)$ is $2$ unless
$2a+b+2\equiv0\pmod p$ in which case the multiplicity is $1$.

\item{\rm (4)} If $G$ is of type $B_n$ with $\lambda=\om_1+\om_n$,
then the weight $\lambda-\alpha_1-\cdots-\alpha_n$ has
multiplicity $n$ in $L_G(\lam)$, unless $p\vert (2n+1)$ in
which case it has multiplicity $n-1$.

\item{\rm (5)} If $G$ is of type $G_2$ and $\lambda=a\om_1+b\om_2$,
with $ab\ne 0$, the multiplicity of the weight
$\lambda-\alpha_1-\alpha_2$ in $L_G(\lam)$ is $2$ unless
$3a+b+3\equiv0\pmod p$ in which case the multiplicity is $1$.

\item{\rm (6)} If $G$ is of type $A_n$ and $\lambda =
a\om_i+b\om_{i+1}+c\omega_{i+2}$, with $abc\ne0$ and
$a+b=p-1=b+c$,
 the weight
$\lambda-\alpha_1-\alpha_2-\alpha_3$ has multiplicity at least
$2$ in the module $L_G(\lam)$.

\item{\rm (7)} If $G$ is of type $D_4$ and $\lambda=a\om_1$, with $a>1$,
then the weight $\lambda-2\alpha_1-2\alpha_2-\alpha_3-\alpha_4$
has multiplicity at least $2$ in $L_G(\lam)$.

\item{\rm (8)} If $G$ is of type $G_2$,
$\lambda=b\om_1$ with $b>1$, and $p>3$, then the weight
$\lambda-2\alpha_1-\alpha_2$ has multiplicity $2$ in
$L_G(\lam)$.

\item{\rm (9)} If $G$ is of type $G_2$,
$\lambda=a\om_2$, with $p>3$ and $a=\frac{p-1}{2}$, then the
weight $\lambda-2\alpha_1-2\alpha_2$ has multiplicity $2$ in
$L_G(\lam)$.
\end{enumerate}
\end{lemma}

\begin{proof} Part (a) is \cite[8.6]{Seitzclass}, (c) and (e) are
\cite[1.35]{Testexc}. For (b), see the proof of
\cite[6.7]{Seitzclass} and apply the main result of \cite{Premet88}. Part (d) is \cite[2.2.7]{BGT}. 
The proof of
(f) is contained in the proof of \cite[6.10]{Seitzclass}. Part (g) is
proved in \cite[6.13]{Seitzclass}. Finally, the proofs of (h) and (i) follow from the proof 
 in \cite[6.18]{Seitzclass}.\end{proof}

We now turn our attention to the determination of the set 
$\Omega_2(G)\setminus \Omega_1(G)$. 
Let $\lambda$ be a $p$-restricted dominant weight for the group $G$.
It will be useful to work inductively, restricting the representation $L_G(\lam)$ to
certain subgroups and applying the following analogue of \cite[6.4]{Seitzclass}.

\begin{lemma}\label{44}  Let 
 $X$ be a subsystem subgroup of $G$ normalized by 
$T$.  Let $\lambda\in \Omega_2(G)$. Let $L_X(\mu)$ be an $FX$-composition factor
of $L_G(\lambda)$, for some dominant weight $\mu$ in the character group of $X\cap T$.
 Then $\mu\in\Omega_2(X)$. 
\end{lemma}

\begin{proof} The argument is completely analogous to the proof of \cite[Lemma 6.4]{Seitzclass}. 
Set $W: = L_G(\lambda)$.
Write $TX=XZ$, where  $Z=C_T(X)^\circ$. Let $0\subset M_1\subset\cdots\subset M_t=W$ be an 
$F(XT)$-composition series of $W$. Then there exists $i$ such that $L_X(\mu) \cong M_i/M_{i-1}$.
Now $M_i = M_{i-1}\oplus M'$ as $FT$-modules, $Z$ acts by scalars on $M'$ and
the set of $(T\cap X)$-weights in $M'$ (and their multiplicities) are precisely the same as in 
$L_X(\mu)$. Also, if $\nu$ is a non-zero weight of $L_X(\mu)$, then $\nu$ corresponds to a non-zero $T$-weight of $M'$. Therefore if $\nu$ is a $(T\cap X)$-weight occurring in $L_X(\mu)$ with
multiplicity greater than $1$, there exists a $T$-weight $\nu'$ such that $\dim(M')_{\nu'}\geq 2$.
So $\nu'=0$ and hence $\nu=0$. The result follows.\end{proof}

\begin{propo}\label{te1} The set of weights $\Omega_2(G)\setminus \Omega_1(G)$
is as given in Table~$\ref{tab:omega2}$. Moreover, the multiplicity of the zero weight in $L_G(\lambda)$ is
as indicated in the fourth column.
\end{propo}

\begin{proof} We first note that for $G$ and $\lambda$ as in Table~\ref{tab:omega2}, 
the multiplicity of the zero weight in $L_G(\lambda)$ 
 can be deduced from 
\cite[Table 2]{Lubeck}. We now show that the list in Table~\ref{tab:omega2} contains
all weights in $\Omega_2(G)\setminus\Omega_1(G)$.
Let $\lambda\in\Omega_2(G)\setminus\Omega_1(G)$; in particular, 
$0$ must be subdominant to $\lambda$, and so $\lam$ lies in the root lattice. 
We will proceed as in
\cite[\S 6]{Seitzclass}. We apply Lemma~\ref{44} to various subsystem subgroups of $G$; all
of such are taken to be normalized by the fixed maximal torus $T$.

\noindent {\bf Case $\mathbf{A_3}$.} Consider first the case where 
$\lambda=b\om_2$. By the above remarks, $b>1$, and so we have $p>2$.
If $b=2$, the only weights subdominant to $\lambda$ are
$\lambda-\alpha_2$, which has multiplicity $1$ in $L_G(\lam)$
and $\mu = \lambda-\alpha_1-2\alpha_2-\alpha_3$ which is the zero
weight. Hence $\lambda\in\Omega_2(G)$ and by Lemma~\ref{wtlem1}(2), $\lambda\in\Omega_1(G)$ 
if and only if $p=3$. If $b>2$, the weight $\mu$ is a
non-zero weight and Lemma~\ref{wtlem1}(2) implies that
$b=p-1$, in which case $\lambda\in\Omega_1(G)$.

Now consider the general case
$\lambda=a\om_1+b\om_2+c\om_3$. Assume for the moment
that $abc\ne0$. Applying Lemma~\ref{wtlem1}(1), we see that
$a+b=p-1=c+d$. But then Lemma~\ref{wtlem1}(6) rules out this
possibility. Hence we must have $abc=0$. If $ab\ne 0$ as above we
have $a+b=p-1$ and $\lambda\in\Omega_1(G)$. The case $bc\ne 0$ is
analogous. If $b=0$ and
$ac\ne 0$,  Lemma~\ref{wtlem1}(1) implies that the weight
$\lambda-\alpha_1-\alpha_2-\alpha_3$ must be the zero weight and
hence $a=1=c$. This weight appears in Table~\ref{tab:omega2}. Finally, if $b=c=0$ or $a=b=0$, then $\lambda\in\Omega_1(G)$.  
This completes the case
$G=A_3$. \smallbreak
\noindent{\bf Case $\mathbf{A_n}$, $\mathbf{n\ne 3}$.} 
If $n=2$, Lemma~\ref{wtlem1}(1) and Proposition~\ref{sz7} give the result. 
So we now assume
$n>3$. Applying Lemma~\ref{44} to various $A_3$ Levi factors of $G$,
as well as the result of Lemma~\ref{wtlem1}(1) and (2), we are
reduced to configurations of the form:
\begin{enumerate}[a)]
\item $\lambda=\om_i$ for some $i$,
\item $\lambda=(p-1)\om_i$, for $1<i<n$,
\item $\lambda=a\om_i$, $i=1,n$,
\item $\lambda=c\om_i+d\om_{i+1}$, $cd\ne0$, $1\leq i<n$ and $c+d=p-1$,
\item $\lambda = \om_1+\om_n$.
\end{enumerate}
Each of these weights is included either in $\Omega_1(G)$ or in Table~\ref{tab:omega2}.
 \smallbreak
\noindent{\bf Case $\mathbf{C_2}$.} Let $\lambda=d\om_1+c\om_2$. The
arguments of \cite[6.11]{Seitzclass} together with
Lemma~\ref{wtlem1}(3) show that
either $\lambda\in\Omega_1(G)$ or one of the following
holds:
\begin{enumerate}[(a)]
\item $d=0$, $c>1$, $c\ne \frac{p-1}{2}$, and the weight $\lambda-2\alpha_1-2\alpha_2$ is the
zero weight.
\item $c=0$, $d>1$, and the weight $\lambda-2\alpha_1-\alpha_2$ is the
zero weight.
\item $cd\ne 0$, $d>1$, 
$2c+d+2\equiv 0\pmod p$ and $\lambda-\alpha_1-\alpha_2$ is the zero weight.
\end{enumerate}
Case (i) is satisfied only if $c=2$; (ii) is satisfied only if
$d=2$; (iii) is not possible. \smallbreak
\noindent{\bf Case $\mathbf{C_n}$, $\mathbf{n>2}$.} First suppose $n=3$ and let
$\lambda=a\om_1+b\om_2+c\om_3$. We apply Lemma~\ref{44} to three
different $C_2$ subsystem subgroups of $G$, namely $X_1$, the
Levi factor corresponding to the set $\{\alpha_2,\alpha_3\}$,
$X_2$, the conjugate of this group by the reflection
$s_1$, and $X_3 = X_1^w$, where $w =
s_1s_2$. Restricting $\lambda$ to $X_1$ gives that 
$$(b,c)\in\{(0,0), (1,0), (2,0), (0,1), (0,2), (1,\frac{p-3}{2}), (0,\frac{p-1}{2})\}.$$
 Note also that $\lambda|_{T\cap X_2}$ has highest weight
$(a+b)\mu_1+c\mu_2$, where $\mu_1,\mu_2$ are the
fundamental dominant weights corresponding
 to the base $\{\alpha_1+\alpha_2,\alpha_3\}$. 

Suppose first that $c\ne 0$,
so $b\in\{0,1\}$.
If $a+b<p$, we can again apply the $C_2$ result to the group $X_2$ to see that 
$a+b=0$ or $1$. On the other hand, if $a+b\geq p$,
we must have $b=1$ and $a=p-1$. This latter case is not possible
as Lemma~\ref{wtlem1}(1) implies that the non-zero weight
$\lambda-\alpha_1-\alpha_2$ has multiplicity $2$. Hence when $c\ne0$,  we have
$(a,b,c) = (0,0,c)$, or $(a,b,c)=(1,0,c)$, or $(a,b,c)=(0,1,\frac{p-3}{2})$.
 As the third possibility corresponds to a weight in $\Omega_1(G)$, 
we consider the first two possibilities. If
$(a,b,c)=(0,0,c)$, by Proposition~\ref{sz7}, we may assume $c\ne \frac{p-1}{2}$, 
and $c\ne1$. This leaves us with the weight $2\omega_3$, and $p\ne 5$. But then \cite{Lubeck} shows that a non-zero weight has multiplicity greater than 1. If $(a,b,c) = (1,0,c)$, 
then $c\in\{1,2,\frac{p-1}{2}\}$. The restriction 
of $\lambda$ to the subgroup $X_2$ is $\mu_1+c\mu_2$. But here Lemma~\ref{wtlem1}(3)
shows that the weight
 $\mu = \lambda-(\alpha_1+\alpha_2)-\alpha_3$ has multiplicity 2 unless 
$c=\frac{p-3}{2}$. Since $\mu$ is a non-zero weight, either $c=1$ and $p=5$ or
 $c=2$ and $p=7$. Again, we refer to \cite{Lubeck} to see that there is a non-zero weight
 with multiplicity greater than one in each case.

Suppose now $c=0$. Then the restriction to $X_1$ implies that $(a,b,c)$
is one of $(a,1,0)$, $(a,2,0)$, $(a,0,0)$. Suppose $(a,b,c)=(a,1,0)$. If $a=0$, 
then the only subdominant
weight in $L_G(\lam)$ is the $0$ weight and hence this gives an
example. If $a\ne0$, Lemma~\ref{wtlem1}(1) implies that $a=p-2$.
Then the restriction of $\lambda$ to $X_3$ is the weight
$a\eta_1+\eta_2$, where $\eta_1,\eta_2$ are the fundamental
dominant weights corresponding to the base
$\{\alpha_1,2\alpha_2+\alpha_3\}$. But then Lemma~\ref{wtlem1}(3)
implies that the non-zero weight
$\lambda-\alpha_1-2\alpha_2-\alpha_3$ occurs with multiplicity
$2$. If $(a,b,c)=(a,2,0)$, then the subdominant weight
$\lambda-2\alpha_2-\alpha_3$ has multiplicity $2$ and
hence this is not an example. Finally, if $(a,b,c)=(a,0,0)$, we
consider the restriction of $\lambda$ to the subgroup $X_3$ and
the $C_2$ result implies that $a=1$ or $a=2$. If $a=1$, then $\lambda\in\Omega_1(G)$, 
while if $a=2$, \cite[Table 2]{Lubeck} shows that $\lambda\in\Omega_2(G)$.  This completes the consideration of the
case $G=C_3$.

Consider now the general case where $n\geq 4$. If $\lambda=\sum a_i\om_i$ with 
$a_i=0$ for
$i\leq n-2$, then the $C_3$ result and Lemma~\ref{44} (applied to the standard $C_3$
Levi factor) implies that either $\lambda$ is one of the weights in 
Table~\ref{tab:omega1} or in Table~\ref{tab:omega2},  or 
$\lambda = \om_{n-1}$ or $\lambda=\om_n$, with $p\ne3$ in each case. 
If $\lambda=\om_{n-1}$, then we refer to \cite{Lubeck}, for the group
$C_4$,
 to see that the
subdominant weight $\om_{n-3}$ occurs with multiplicity $2$. This then shows that 
$\lambda=\om_{n-1}\not\in\Omega_2(G)$. Now if $\lambda = \om_n$, again use \cite{Lubeck} and find that
$\lambda\in\Omega_2(G)$
 when $n=4$, while if $n>4$, the weight 
$\om_{n-4}$ occurs with multiplicity 2 and so $\lambda\not\in\Omega_2(G)$.

We may now assume that there exists $i\leq n-2$ with $a_i\ne 0$.
Choose $i\leq n-2$ maximal with $a_i\ne 0$ and consider the $C_3$ subsystem
subgroup $X$ with root system base $\{\alpha_i+\cdots +\alpha_{n-2},
\alpha_{n-1},\alpha_n\}$, so that $\lambda|_{T\cap X} =
a_i\eta_1+a_{n-1}\eta_2+a_n\eta_3$, where
$\{\eta_1,\eta_2, \eta_3\}$ are the fundamental dominant
weights corresponding to the given base. As $a_i\ne 0$, the $C_3$
case considerations imply that $a_{n-1}+a_n=0$ and $a_i=1$ or $2$. If $i=1$,
$\lambda$ occurs in the statement of the result. If $i=2$, so
$\lambda=a_1\om_1+a_2\om_2$, then we may assume $a_1\ne
0$, or $a_2=2$, as $\lambda=\om_2$ occurs in the statement of
the result. But then the restriction of $\lambda$ to the $C_3$
subsystem subgroup with root system base
$\{\alpha_1,\alpha_2+\cdots+\alpha_{n-1},\alpha_n\}$ has
non-zero weights occurring with multiplicity greater than $1$.

So finally, we may assume $i>2$, and so $n\geq5$. If $n=5$ and so
$i=3$, we apply the result for $C_4$ to see that $a_2=0$ and $a_3=1$. But the non-zero
weight $\lambda-\alpha_2-2\alpha_3-2\alpha_4-\alpha_5$ has
multiplicity at least $2$ by the $C_4$ result. Hence the result holds for the case
$n=5$. If $n\geq 6$, we consider the $C_5$ subsystem subgroup
whose root system has base
$\{\alpha_{i-2},\alpha_{i-1},\alpha_i+\cdots+\alpha_{n-2},\alpha_{n-1},\alpha_n\}$
and obtain a contradiction.\smallbreak

\noindent{\bf Case $\mathbf{D_n}$.} Suppose first that $n=4$. Let $\lambda=\sum a_i\om_i$. If $a_2=0$,
Lemma~\ref{wtlem1}(1) implies that $\lambda=a_i\om_i$ for
$i=1$, $3$ or $4$; assume by symmetry that $i=1$. Then
Lemma~\ref{wtlem1}(7) shows that either $a_1=1$ or
$\lambda-2\alpha_1-2\alpha_2- \alpha_3-\alpha_4$ is the zero
weight. In either case, $\lambda$ is as in the statement of the result.
So we now assume $a_2\ne 0$. Using the result for $A_3$, applied
to the three standard $A_3$ Levi factors, we see that at most one
of $a_1,a_3,a_4$ is non-zero.
 However if $a_1+a_3+a_4\ne 0$, Lemma~\ref{wtlem1}(1) and (2) provide a
contradiction. Hence
$\lambda=a_2\om_2$, and since $\lambda=\omega_2$ is in Table~\ref{tab:omega2}, we
may assume $a_2>1$. Now $\lambda-\alpha_1-2\alpha_2-\alpha_3$ is a non-zero weight,
so Lemma~\ref{wtlem1}(2) implies that $a_2=p-1$. But we now apply Lemma~\ref{44}
to the $A_3$ subsystem
subgroup with root system having as base
$\{\alpha_2,\alpha_1,\alpha_2+\alpha_3+\alpha_4\}$ to obtain a contradiction.

Now consider the general case where $n>4$. We argue by induction
on $n$. Apply the result for $D_{n-1}$ to the standard $D_{n-1}$
Levi factor of $G$ to see that either $\lambda$ is as in the
statement of the result or $\lambda = a\om_1$,
$a\om_1+\om_2$, $a\om_1+\om_{n-1}$,
$a\om_1+\om_n$, $a\om_1+\om_3$, or
$a\om_1+2\om_2$. Now consider the $D_4$ subsystem
subgroup whose root system has base
$\{\alpha_1,\alpha_2+\cdots+\alpha_{n-2},\alpha_{n-1},\alpha_n\}$.
The result for $D_4$ then implies that either $\lambda$ appears in
 Table~\ref{tab:omega1} or Table~\ref{tab:omega2}, 
or $\lambda = \om_3$. But then the restriction of $\lambda$ to
the standard $D_{n-1}$ Levi factor $X$ affords a composition factor
which is the unique nontrivial composition factor of ${\rm Lie}(X)$. The weight corresponding to the zero weight in
${\rm Lie}(X)$ is the weight $\lambda-\alpha_1-2(\alpha_2+\cdots+\alpha_{n-2})-\alpha_{n-1}-\alpha_n$,
 which is a non-zero weight in
$L_G(\lam)$ with multiplicity at least $n-3$. (See \cite[Table 2]{Lubeck}.) This completes
the consideration of type $D_n$.\smallbreak

\noindent{\bf Case $\mathbf{B_n}$.} As throughout the paper, we assume $p>2$ when $G$ 
has type $B_n$. Consider first the
case $n=3$. We apply the result for $C_2$ to the standard $C_2$ Levi factor
of $G$ and the $A_3$ result to the subsystem subgroup $X$ with
root system base $\{\alpha_2+2\alpha_3, \alpha_1,\alpha_2\}$. Note that if $\lambda = a\om_1+b\om_2+c\om_3$, 
then the restriction of $\lambda$ to $X$  affords a composition
factor with highest weight $(b+c)\eta_1+a\eta_2+b\eta_3$, where $\{\eta_1,\eta_2,\eta_3\}$ is the set of 
fundamental dominant weights dual to the given base.  We deduce
that either $\lambda$ appears in Table~\ref{tab:omega1} or in Table~\ref{tab:omega2} or
$\lambda = \om_1+\om_3$, $2\om_3$, 
$(p-3)\om_1+2\om_3$, or $(p-2)\om_1+\om_3$. The first case is ruled out by
Lemma~\ref{wtlem1}(4). For the second and third, where $p>2$, we see that the
restriction of $ \lambda$ to the standard $C_2$ Levi factor affords a
composition factor isomorphic to the Lie algebra of the Levi
factor, in which the non-zero weight $\lambda-\alpha_2-2\alpha_3$
has multiplicity $2$. For the final case, when $\lambda = (p-2)\om_1+\om_3$,
 consider the $C_2$ 
subsystem subgroup $X^{s_1}$,
with root system base $\{\alpha_1+\alpha_2,\alpha_3\}$, for which $\lambda$ affords a 
composition factor with
highest weight $(p-2)\zeta_1+\zeta_2$ (where $\zeta_1,\zeta_2$ are the fundamental dominant weights with respect to the 
 base $\{\alpha_1+\alpha_2,\alpha_3\}$), 
contradicting Lemma~\ref{44}. This completes the consideration of $G=B_3$.

Consider now the general case $n\geq 4$. Let $\lambda=\sum a_i\omega_i$. 
By considering the restriction of $\lambda$ to the
standard $B_3$ Levi subgroup, we see that $a_{n-1}+a_n<p$. Now
consider the maximal rank $D_n$ subsystem subgroup $W$ with root
system base
$\{\alpha_1,\dots,\alpha_{n-2},\alpha_{n-1},\alpha_{n-1}+2\alpha_n\}$.
The above remarks imply that $\lambda|_{T\cap X}$ is a
$p$-restricted weight and the  result for $D_n$ then gives the
result for $B_n$.\smallbreak

\noindent{\bf Case $\mathbf{E_n}$.} For $n=6$, we apply
the $D_n$ result to the two standard $D_5$ Levi subgroups of $G$
and the $A_n$ result to the standard $A_5$ Levi subgroup
 of $G$ to see that either $\lambda$ is as in Table~\ref{tab:omega2}
or $\lambda=\om_3$, $\om_5$, $2\om_1$ or $2\om_6$.
In the first two cases, the restriction of $\lambda$ to one of the
$D_5$ Levi factors affords a composition factor which is 
isomorphic to the unique nontrivial composition factor of
the Lie algebra of the Levi factor. But the zero
weight in this composition factor corresponds to a non-zero weight
with multiplicity at least $4$. In the last two configurations, we
note that the $D_5$ composition factor afforded by $\lambda$ has
zero as a subdominant weight of multiplicity at least $3$ (here we
use \cite{luebeck}). But this subdominant
weight is a non-zero weight with respect to $T$.

Now for $n=7$, we apply the result for $E_6$ as well as for
$D_{6}$ and $A_{6}$ to see that either $\lambda$ is as in the
statement of the result or $\lambda = \om_6$ or $2\om_7$.
These two configurations can be ruled out exactly as in the case
of $E_6$. The case of $G=E_8$ is completely analogous.\smallbreak

\noindent{\bf Case $\mathbf{F_4}$.} For this case, we
use the standard $B_3$ and $C_3$ Levi factors and the maximal rank
$D_4$ subsystem subgroup whose root system base is
$\{\alpha_2+2\alpha_3+ 2\alpha_4,
\alpha_2,\alpha_1,\alpha_2+2\alpha_3\}$. This leads immediately to
the result.\smallbreak

\noindent{\bf Case $\mathbf{G_2}$.} Let
$\lambda=b\om_1+a\om_2$, where
$3(\alpha_1,\alpha_1)=(\alpha_2,\alpha_2)$. We first treat the
cases where $p=2$ or $p=3$ by referring to \cite{luebeck} to see
that $\lambda$ is as in Table~\ref{tab:omega2} or
$\lambda=2\om_1+2\om_2$ and $p=3$. But then
 Lemma~\ref{wtlem1}(5) shows that the non-zero weight $\lambda-\alpha_1-\alpha_2$ has
 multiplicity $2$. So we may now assume $p>3$. We will consider the restriction
of $\lambda$ to $X$, the $A_2$ subsystem subgroup corresponding to
the long roots in $\Phi(G)$; we have $\lambda|_{T\cap X} =
(a+b)\eta_1+a\eta_2$, where $\{\eta_1, \eta_2\}$ are the
fundamental dominant weights for $X$.

Assume that $\lambda$ is not as in Table~\ref{tab:omega2}, so
$\lambda\ne \om_i$, $i=1,2$. Now Lemma~\ref{wtlem1}(8) implies
that $a\ne 0$. Consideration of the action of $X$, together with
Lemma~\ref{wtlem1}(1), implies that either $a+b\geq p$ or
$2a+b+1\equiv 0\pmod p$. Now if $b=0$, we must be in the second
case, so $a=\frac{p-1}{2}$; in particular, $a>1$. But then
Lemma~\ref{wtlem1}(9) gives a contradiction. Hence we must have
$b\ne 0$. Then Lemma~\ref{wtlem1}(5) implies that
$3a+b+3\equiv0\pmod p$ and either $2a+b+1\equiv0\pmod p$ or $a+b\geq
p$. In the first case we have $a=p-2$ and $b=3$; in the second
case we must have $b\geq 2$. So in either case
$\lambda-2\alpha_1-\alpha_2$ is a subdominant weight. Arguing as
in \cite[6.18]{Seitzclass}, we see that the non-zero weight $\lambda-2\alpha_1-\alpha_2$ has
multiplicity $1$ only if $6a+4b+8\equiv0\pmod p$, which together
with the previous congruence relation implies that $b=p-1$, and
hence $3a+2\equiv0\pmod p$. We are now in the situation where
$a+b\geq p$; indeed, as $a=\frac{p-2}{3}$ or $\frac{2p-2}{3}$, we
have that either $p=5$, $a=1$, $b=4$ and \cite{luebeck} gives a
contradiction, or $\lambda|_{T\cap X} = p\eta_1 +
(c\eta_1+d\eta_2)$, with $cd\ne0$, $(c,d)=
(\frac{p-5}{3},\frac{p-2}{3})$ or
$(\frac{2p-5}{3},\frac{2p-2}{3})$. In each case, the result for
$A_2$ leads to a contradiction.

It remains to verify for each weight $\lambda$ in Table~\ref{tab:omega2} that all
non-zero weights of $L_G(\lambda)$ do indeed have multiplicity 1. This
is straightfoward using \cite{Lubeck} and \cite{luebeck}.\end{proof}

The following corollary is immediate.

\begin{corol} Let $\lambda\in X(T)$. If all non-zero weights of $L_G(\lambda)$ occur with multiplicity $1$,
then the zero weight occurs with multilicity at most ${\rm rank}(G)$.
\end{corol}

Note that lower bounds for the maximal weight multiplicities in irreducible representations of $G$ are studied in \cite{BOS1,BOS2}. The above corollary does not however follow from their results.

\begin{corol}\label{dt1} Let $\lambda\in\Omega_2(G)$. Then
one of the following holds:\begin{enumerate}[]
\item{\rm (1)} $\lambda\in\Omega_1(G)$. 
\item{\rm (2)} $L_G(\lam)$ is the unique nontrivial composition factor of ${\rm Lie}(G)$.
\item{\rm (3)} $(G,\lambda)$ is one of $(A_3,2\om_2)$, $(B_n,2\om_1)$,
$(C_n,\om_2)$, $(C_2,2\om_2)$, $(C_4,\om_4)$, $(D_n,2\om_1)$,
$(F_4,\om_4)$, where $p\neq 3$ in the latter case.
\end{enumerate}
In addition, the zero weight \mult in cases $(2),(3)$ is as in
Table~$\ref{tab:omega2}$.
\end{corol}

\begin{proof} The zero
weight \mult
 can be easily computed from the knowledge of the dimension of $L_G(\lam)$, provided in
\cite{Lubeck}. The statement about ${\rm Lie}(G)$ follows 
from the known  structure of the $FG$-module ${\rm Lie}(G)$; see for example 
\cite[1.9]{Seitzclass}.\end{proof} 

We record in the following corollary the cases where the 0 weight has 
multiplicity 2
in $L_G(\lam)$, for $\lambda\in\Omega_2(G)$.
This will be required for the resolution of Problem 1, in case $H = D_n$.

\begin{corol}\label{dt3} Let $\lambda\in\Omega_2(G)$ 
and suppose that the zero weight has multiplicity $2$ in
$L_G(\lam)$. Then the pair $(G,\lambda)$ is as in Table~$\ref{tab:mult2}$.

\begin{table}
$$\begin{array}{|l|c|c|c|}
\hline
G& \lambda&\mbox{conditions}&\dim L_G(\lambda)\\
\hline
A_2&\omega_1+\omega_2&p\ne 3&8\\\hline
A_3&2\omega_2&p>3&20\\
&\omega_1+\omega_3&p=2&14\\
\hline
C_2&2\omega_1&&10\\
&2\omega_2&p\ne 5&14\\
\hline
C_3&\omega_2&p\ne 3&14\\
\hline
C_4&\omega_2&p=2&26\\
&\omega_4&p\ne 2,3&42\\
\hline
D_4&\omega_2&p=2&26\\
\hline
F_4&\omega_1&p=2&26\\
&\omega_4&p\ne 3&26\\
\hline
G_2&\omega_2&p\ne 3&14\\
\hline
\end{array}$$
\caption{$\lambda\in\Omega_2(G)$, 0 weight in $L_G(\lambda)$ of multiplicity 2}\label{tab:mult2}
\end{table}

\end{corol}

We can now complete the proof of Theorem~\ref{mt11}.

\begin{theo}\label{w1t} Let $\lambda\in X(T)$ be a non-zero dominant weight. Then
at most one weight space of $L_G(\lam)$ is of dimension greater than one
\ii one of the \f holds:
\begin{enumerate}[]
\item{\rm (1)} the module $L_G(\lambda)$ is as described in Proposition~$\ref{sz7}$;
\item{\rm (2)} $\lambda=p^k\mu$ for some $k\in{\mathbb N}$, $k=0$ if ${\rm char}(F)=0$, and for some weight 
$\mu\in\Omega_2(G)\setminus \Omega_1(G)$, as given in Table~$\ref{tab:omega2}$.
\end{enumerate}
\end{theo}

\begin{proof} By Propositions~\ref{sz7} and \ref{te1}, the modules $L_G(\lambda)$ described in (1) and (2), all
non-zero weights are of multiplicity at most 1. So now take $\lambda\in X(T)$ a non-zero dominant weight
such that at most one weight space of $L_G(\lambda)$ has dimension greater than 1, and suppose $\lambda$ is not as in (1).
In particular, $L_G(\lambda)$ has a weight of multiplicity greater than 1.
Note that if a non-zero weight in $L_G(\lambda)$ occurs with multiplicity greater than 1, then
so do all of its conjugates under the Weyl group. Hence, the weight
occurring with multiplicity greater than one in $L_G(\lambda)$  must be the zero weight.

Now, let $\lambda = \sum_{i=1}^l p^{k_i}\lambda_i$, where $\lam_i$ is a non-zero dominant $p$-restricted weight for all $i$; 
so we have
$L_G(\lambda)\cong L_G(p^{k_1}\lam_1)\otimes\cdots \otimes L_G(p^{k_l}\lam_l)$.
If $l=1$ then (2) holds. Let $l>1$. Then for all $1\leq i\leq l$, 
the weights of $L_G(p^{k_i}\lam_i)$ must have multiplicity 1 (else a non-zero weight has
 multiplicity greater than 1 in $L_G(\lambda)$). Then by Proposition~\ref{sz7},
we see that there exists a pair 
$(\lambda_i,\lambda_j)$ with $k_j=k_i+1$, as in Proposition~\ref{sz7}(2)(b). So
we investigate the multiplicity of the zero weight in the 2-fold tensor products associated to 
these pairs of weights.  

Let $L_G(\lambda)=L_G(\lam_1)\otimes L_G(p\lam_2)$, where
$(G,p,\lam_1,\lam_2)$ are the following: (a)\ $(C_n,2,\om_n,\om_1)$, (b)\ $(G_2,2,\om_1,\om_1)$,
(c)\ $(G_2,3,\om_2,\om_1)$.

In case (a), let $\ep_i$ be the weights defined as in \cite[Planche III, IX]{Bourb4-6}.
The weights of $L_G(\om_n)$ are $\pm\ep_1\pm \cdots \pm \ep_n$
and the weights of $L_G(2\om_1)$ are $\pm 2\ep_1\ld \pm 2\ep_n$;
in particular, there are no common weights. It follows that the 0 weight does
not occur in $L_G(\lam_1)\otimes L_G(2\lam_2)$. 

In case (b), the weights of $L_G(\om_1)$ are $\{\pm (\ep_1-   \ep_2), \pm
(\ep_1- \ep_3), \pm (\ep_2- \ep_3)\}$. So $L_G(\om_1)$ and
$L_G(2\om_1)$ again have no common weight, and hence the 0 weight does not occur in $L_G(\om_1)\otimes
L_G(2\om_1)$. 

Finally for case (c), the weights of $L_G(3\om_1)$ are $\{0, \{\pm 3(\ep_1-   \ep_2),
\pm 3(\ep_1- \ep_3), \pm 3(\ep_2- \ep_3)\}$, and the weights of
$L_G(\om_2)$ are $0, \pm (2\ep_1-\ep_2- \ep_3), \pm (2\ep_3-\ep_1-
\ep_2), \pm (2\ep_2-\ep_1- \ep_3)$. Then the \mult of the weight 0
in
 $L_G(\lam_1)\otimes L_G(3
\lam_2) $ is 1.

It follows from Proposition \ref{sz7} that the weights
of $L_G(\lam_1)\otimes L_G(p\lam_2)$ whose \mult is greater than 1
are non-zero. Hence there are no examples of $\lambda= \sum_{i=1}^l p^{k_i}\lambda_i$ with $l>1$ and all $\lambda_i$ 
different from 0 such that the zero weight has multiplicity greater than 1 and all non-zero weights have multiplicity 1
in the module $L_G(\lambda)$. \end{proof}


\section{Semisimple regular elements in classical groups}\label{sec:sscl}
Throughout this section, $H$ is a simply connected simple algebraic group of classical type, with 
natural module $V$. Let $\tilde H$ be the image of $H$ in ${\rm GL}(V)$. Note that 
$x\in H$ is regular if and only if the image of $x$ in $\tilde H$ is regular in $\tilde H$. Moreover, we will assume $p>2$ when $H$ is of type $B_m$, as in the case $p=2$,
it will suffice to treat the group of type $C_m$.

\begin{lemma}\label{2h2} Let $H$ be of type $B_m$, $C_m$ or $D_m$. Let $t\in {\tilde H}$ be a semisimple
regular element.
Then $t$ is a regular element in ${\rm GL}(V)$, except for the \f
cases:\begin{enumerate}[]
\item{\rm (1)} $H$ is of type $B_m$  and $-1$ is an \ei of $t$ on $V$
with \mult $2$;
\item {\rm (2)} $H$ is of type $D_m$  and either $1$ or $-1$ or
both are \eis of $t$ on $V$ with
\mult $2$.
\end{enumerate}
\end{lemma}

\begin{proof} Let $d=\dim V$ and let $b_1\ld b_d$ be a basis in $V$ with respect to which the Gram matrix of the bilinear form
is anti-diagonal, with all non-zero entries in the set $\{1,-1\}$. We may assume the matrix of $t$ with respect to this basis is diagonal and of the form
$t=\diag (t_1\ld t_m,x,t_m\up\ld t_1\up )$, where $x$ is
absent if $d$ is even and equal to 1 otherwise. Since
$t$ is regular in $H$, $t_1\ld t_m$ are distinct. Moreover,
if $H$ is of type $C_m$ then all $t_1\ld t_m,t_m\up\ld t_1\up $  are distinct. Indeed,
 the roots of $C_m$ take values $\{t_it_j^{-1}, t_it_j, t_j^2\ |\ 1\leq i\ne j\leq m\}$ and $t$ lies in the kernel of no root. 
If $H$
is of type $B_m$, $m>2$, then the roots of $H$ take values $\{t_it_j^{-1}, t_it_j, t_j\ |\ 1\leq i\ne j\leq m\}$ on $t$. So no $t_i$  is equal to $ 1$ and therefore
the eigenvalues $t_1\ld t_m,x,t_m\up\ld t_1\up$ are distinct except for the case
(1). Finally, for $H$ of type $D_m$, we argue as above, using the fact that the roots of $H$ take values $\{ t_it_j^{-1}, t_it_j\ |\ 1\leq i\ne j\leq m\}$ on $t$, which leads to  (2). \end{proof}

\begin{lemma}\label{33}
 Let $H$ be of type $B_m$, $C_m$ or $D_m$, and let
$T'$ be a regular torus in $\tilde H$.
Then $T'$ is a regular torus in ${\rm GL}(V)$, except for the
case where $H$ is of type $D_m$  and the fixed point
subspace of $T'$ on $V$ is of dimension  $2$.\end{lemma}

\begin{proof} Appying Lemma~\ref{ca1}, we choose $t\in T'$ such that $C_H(t)^\circ = C_H(T')$. 
By Lemma~\ref{2h2}, $t$ is regular in ${\rm GL}(V)$ if $H$ is of type $C_m$,
and so the result holds. If $\dim V$ is odd, that is, if $H$ is of type $B_m$,
let $T$ be a maximal torus of $H$.
Then for all $T$-weights $\mu$, $\nu$ of $V$, the difference $\mu-\nu$ is a multiple of a root of $H$ (with respect to $T$).
Since $T'$ is a regular torus in $\tilde H$, $T'\not\subset{\ker}(\beta)$ for any root $\beta$ of $H$. Hence 
$T'$ has distinct weights
on $V$, and so is a regular torus in ${\rm GL}(V)$.

Finally, consider the case $H = D_m$. Let $b_1\ld b_{2m}$ be a basis for $V$ as in the previous proof.
We may assume that $T'$ consists of diagonal matrices with respect to this basis.
Now, $T'$ regular implies that the weight of $T'$ afforded by $\lan b_j\ran $ is distinct from the weight afforded by 
$\lan b_{2m-k+1}\ran$ for all $k\not\in\{ j, 2m-j+1\}$. So the weights of $T'$ on $V$ are distinct unless 
there exists $j$, $1\leq j\leq m$, such that the weight afforded by $\lan b_j\ran$ is equal to that afforded by $\lan b_{2m-j+1}\ran$. So assume these two weights coincide. Since the weights afforded by these two vectors differ by a sign, if they are equal, they must be 0. 
Hence there is a unique such $j$ and the result holds.\end{proof}

\begin{propo}\label{55y} Let $H$ be of type $D_m$ with $m\geq 4$.
 Let 
$T$ be a torus in $\tilde H$.
 Suppose that there exists at most one $T$-weight of $V$ of multiplicity greater than $1$, and if such a weight exists,
its multiplicity is $2$.
Then
$C_{\tilde H}(T)$ is a maximal torus in $\tilde H$, that is, $T$ is a regular torus 
in $\tilde H$.
\end{propo}

\begin{proof} If all  weights of $T$ on $V$ have \mult 1, then $T$ is a regular torus in ${\rm GL}(V)$,
and hence in $\tilde H$. Suppose that there is a weight $\mu$ of $T$ on $V$ of \mult 2, and let $M$ be the corresponding weight space. Denote by $R$  the set  of singular vectors in 
$ M$ together with the zero vector.

If $R=M$ then  $M$ is totally singular, $M\subset M^\perp$, and $V|_{T}=M^{\perp}\oplus V_2$. It is well-known that the $FT$-module
$V_2$ is dual to $M$, and hence is contained in a $T$-weight space.
This is a contradiction as $\dim V_2=2$ and by hypothesis $M$ is the unique weight space of dimension greater than 1.

Suppose that $M$ is neither  totally singular nor non-degenerate with respect to the bilinear form on $V$. Then 
$R$ is a 1-dimensional $T$-invariant subspace of $M$ and $M\subset R^\perp$. Let $M|_T=R\oplus R'$ for a $T$-submodule $R'$. Then non-zero vectors  $x\in R'$ are non-singular, and if $t\in T$ then  $tx=\mu(t) x$.
Let $Q$ be the quadratic form on $V$ preserved by $\tilde H$. Then $Q(x)=Q(tx)=\mu(t)^2Q(x)$, for all $t\in T$, whence $\mu=0$. 
Therefore, $T$ acts trivially on $M$, and hence on $R$. However,
$V/R^\perp$ is an $FT$-module dual to $R$, so $T$ acts trivially on $V/R^\perp$. As $V$ is a
completely reducible $FT$-module, it follows that the dimension of the zero $T$-weight space on 
$V$ is at least 3, which contradicts the hypothesis.

Finally, suppose that $M$ is non-degenerate with respect to the bilinear form on $V$. 
The reductive group $C_{\tilde H}(T)$ stabilizes all $T$-weight spaces of $V$. Since
${\rm SO}_2(F)$ has no unipotent elements,  $C_{\tilde H}(T)$ must be 
  a torus,  as required.\end{proof}

\begin{proof}[Proof of Proposition~$\ref{55yy}$.] The Proposition
 follows directly from Lemma~\ref{33} and Proposition~\ref{55y}. \end{proof}

As explained in Section~\ref{sec:intro} and Section~\ref{sec:red}, for a classical group $H$ with natural
module $V$, Proposition~\ref{55yy}  reduces the resolution of Problems 1 and 2
 to a question about weight multiplicities in the module $V|_G$, for a simple algebraic
group $G$ acting irreducibly and tensor-indecomposably on $V$. 
Given  a simple subgroup $G$ of $H$, we see that a maximal torus
$T$ of $G$ is a regular torus of $H$ if and only if one of the following holds:
\begin{enumerate}[a)]
\item all $T$-weight spaces of $V$ are 1-dimensional, and so $V$ is described by 
Proposition~\ref{sz7}, or 
\item $G\subset H=D_m$ and the 0 weight space of $V$ (viewed as a $T$-module) is 2-dimensional, while all other $T$-weight spaces 
of $V$ are 1-dimensional, so $V$ is described by Theorem~\ref{w1t} and Proposition~\ref{te1}.
\end{enumerate}


\section{Orthogonal and symplectic \reps}\label{sec:forms}

As in the previous sections, we take $G$, $T$ and the rest of the notation to be as fixed 
in Section~\ref{sec:intro}. In this section we partition the irreducible 
\reps $\rho$ of $G$ with highest
weight in $\Omega_1(G)$ into four families depending on whether $\rho(G)$ is contained in a  
 group of type $B_m$,  $C_m$, $D_m$ or in
none of them. (It is well-known that the latter holds \ii
the associated $FG$-module is not self-dual.) 
In addition, we determine which of the weights in $\Omega_2(G)$ correspond
to a representation whose image contains a regular torus of $H = D_m$. This information is collected in
Tables~\ref{tab:Omega1-symp},~\ref{tab:omega2-even-orth},~\ref{tab:Omega1-odd-orth},~\ref{tab:orthp=2},~\ref{tab:nonselfdual},~\ref{tab:p=2symp},
 and completes the proof of Theorem~\ref{mt22}. For simplicity, we will
say that an irreducible representation of $G$ (or the corresponding module $V$) is 
\emph{symplectic}, 
respectively 
\emph{orthogonal}, if $G$ preserves a non-degenerate alternating form, respectively a 
non-degenerate quadratic form on $V$.
As discussed in Section~\ref{sec:red}, for our application, it suffices to consider tensor-indecomposable
modules.

Until the end of the section  $\lam\in X(T)$ is  assumed to be a
$p$-restricted dominant weight. Recall that the highest weight of the irreducible $FG$-module
$L_G(\lambda)^*$ (the dual of $L_G(\lambda$)), is $-w_0\lambda$, where $w_0$ is the longest word in the Weyl 
group of $G$.  Since $-w_0={\rm id}$ for the groups 
$A_1$, $B_n$, $C_n$, $D_n$, 
($n$ even), $E_7$, $E_8$, $F_4$, and $G_2$,
all irreducible modules are self-dual for these groups. (See \cite[16.1]{MTbook}.) The following result treats the remaining cases.

\begin{propo}\label{w2t} Let $G$ be of type $A_n$, $n>1$, $D_n$, $n$ odd, or $E_6$. Let 
$\lam\in \Omega_2(G)$. 
Then $L_G(\lam)$ is a self-dual $FG$-module \ii one of the \f
holds:\begin{enumerate}[]
\item{\rm (1)} $G= A_n$, $n>1$, and either\begin{enumerate}[]
\item{\rm (i)} $\lam= \om_1+\om_n$, or 
\item{\rm (ii)} $\lam = \om_{(n+1)/2}$, for $n$ odd, or
\item{\rm (iii)} $\lam = (p-1)\om_{(n+1)/2}$, for $n$  odd, or
\item{\rm (iv)} $\lam = \frac{p-1}{2}(\om_{n/2}+\om_{(n+2)/2})$, for $n$ even and $p$ odd, or
\item{\rm (v)} $\lam =2\om_2$ for $n=3.$
\end{enumerate}
\item{\rm (2)} $G\cong D_n$, $n$ odd, and $\lam\in\{ \om_1,\om_2,
 2\om_1\}$;
\item{\rm (3)} $G\cong E_6$ and $\lam=\om_2.$
\end{enumerate}
\end{propo}

\begin{proof} The above remarks imply that $L_G(\lam)$ is self-dual if and only if 
$\lam=-w_0\lam$, that is, 
$\lam$ is
invariant under the graph \au of $G$. We refer to Table~\ref{tab:omega2} for the weights $\lam\in\Omega_2(G)$.
The cases
where $G$ is of type $E_6$ or $D_n$ are trivial.
For $G=A_n,n>1$, consider $\om=c\om_i+(p-1-c)\om_{i+1}$, where $0\leq c<p$.
The graph \au sends $\om_j$ to $\om_{n-j+1}$. If $n$ is even, it follows that
$i+1=n-i+1$ and $c=p-1-c$. So $i=n/2$ and $c=(p-1)/2$. Similarly, for
$n$ odd we get $\om=(p-1)\om_{(n+1)/2}$.\end{proof}

\begin{lemma}\label{tk3} \begin{enumerate}[]
\item{\rm (1)} Let $G=A_n$, $p>2$ and let $\lam=(p-1)\om_{(n+1)/2}$ for $n$
 odd, and
$\lam =\frac{p-1}{2}(\om_{n/2}+\om_{(n+2)/2})$ for $n$ even. Then $\dim L_G(\lam)$ is odd.
\item{\rm (2)} Let $G = A_n$, $n$ odd, and $\lambda = \om_{(n+1)/2}$. Then 
$\dim L_G(\lam) = \begin{pmatrix}n+1\cr (n+1)/2\end{pmatrix}$
and so $\dim L_G(\lam)$ is even.
\item{\rm (3)} Let $G=C_n$, $n>1$, $p>2$ and let 
$\lam_1=\frac{p-1}{2}\om_n$, $
\lam_2=\om_{n-1}+\frac{p-3}{2}\om_{n}$. Then 
$\dim L_G(\lam_1)=(p^n+1)/2$ 
and $\dim L_G(\lam_2)=(p^n-1)/2$.
\end{enumerate}
\end{lemma}

\begin{proof} (1) Let $B: = \{(c,i)\mid 0\le c\leq p-1, 0\leq i\leq n\}$ and 
for the purposes of this proof set $\om_0$ and $\om_{n+1}$ to be the 0 weight.
 Then by \cite[Prop. 1.2]{SZ90} (and the discussion on page 555 
{\it{loc.cit})},
 the direct sum of all
\ir \reps $L_G(\mu)$ with $\mu$ running over the set $\{(p-1-c)\om_i+c\om_{i+1}\mid(c,i)\in B\}$ has 
dimension $p^{n+1}$. 
Note that the trivial \rep of $G$ occurs twice among the $L_G(\mu)$ (specifically, for 
$(c,i)=(0,0)$ and $(p-1,n)$).
We also observe that $L_G(\lam)$ with $\lam$ as in (1) is the only nontrivial self-dual 
module among the $L_G(\mu)$,
whereas the other  $L_G(\mu)$ (with $\mu\neq \lam,0$)
 occur in the sum as dual pairs. Therefore, the parity of $\dim L_G(\lam)$ coincides with that of $p^{n+1}$,
which is an odd number.

(2) The irreducible $FG$-module $L_G(\omega_j)$ is the $j$-th exterior power of the 
natural $FG$-module,
whence the result.

(3) The assertion
is proven in \cite{SZ1}, see the statement A of the Main Theorem. \end{proof}

\begin{propo}\label{ty5} Let 
$\lam\in \Omega_2(G)$.  Assume moreover that $p>2$. 
Then $L_G(\lam)$ is symplectic \ii the pair $(G,\lam)$ is as in Table{\rm~\ref{tab:Omega1-symp}}. In particular, if $\lambda\in\Omega_2(G)$ with $L_G(\lambda)$ symplectic, then $\lambda\in\Omega_1(G)$.
\end{propo}

\begin{proof} First, by \cite[Lemma  79]{StYale}, if $|Z(G)|$ is odd then
every self-dual \irr of $G$ is orthogonal; in particular, this is the case
if $G=A_n$ with $n$ even. So we assume that $|Z(G)|$ is even. It
follows that $G$ is classical or of type $E_7$. Furthermore, in the
adjoint \rep of $G$  the center acts trivially, so again by
\cite[Lemma  79]{StYale} the \reps arising from the adjoint \rep
are orthogonal. More generally, any representation where all weights are roots
must be orthogonal. Now let $h\in Z(G)$ as defined in \cite[Lemma 79]{StYale}. 

Using the description of the sum of the positive roots in $\Phi(E_7)$ given in \cite[Planche VI]{Bourb4-6},
we deduce that $h$ acts nontrivially on $L_{E_7}(\om_7)$, and hence by \cite[Lemma  79]{StYale} 
$L_{E_7}(\om_7)$ is symplectic. So we are left
with the classical groups.

Let $G= A_n$ with $n$ odd. One checks that $h$ is the central involution in $G$. 
In view of Proposition~\ref{w2t}, Lemma~\ref{tk3}, and the above
comments, we must consider the cases $\lam=2\om_2$ for $n=3$ and $\lam=\om_{(n+1)/2}$ with $n$ odd. 
In the former case,
$h$ acts trivially on $L_G(2\om_2)$, so
the module is orthogonal. In the second case, $L_G(\lam)$ 
is the $(n+1)/2$-th exterior power of the natural $FG$-module, and so $h$ acts as $(-1)^{(n+1)/2}\cdot{\rm Id}$ on $L_G(\lam)$.
Then \cite[Lemma  79]{StYale} gives the result.  

Let $G= C_n$, $n>1$. Since $L_G(\om_1)$ is the natural symplectic module for $G$,  $h$ acts as 
$-{\rm Id}$
on $L_G(\om_1)$. As $Z(G)$ is of order 2, $L_G(\lam)$ is symplectic \ii $G$ acts faithfully on
 $L_G(\lambda)$.
 As the root lattice $\Z\Phi(G)$ has index 2 in the 
weight lattice, it follows that $G$ acts faithfully on $L_G(\lam)$ \ii $\lam\not\in\Z\Phi(G)$. 
The weights $\lam=\om_2,2\om_1,2\om_2$ all lie in $\Z\Phi(G)$. Note that $\om_n$ (respectively, $\om_{n-1}$) lies in 
$\Z\Phi(G)$ \ii $n$ is even (respectively, odd). (See \cite[]{Bourb4-6}.) Therefore,
 $\frac{p-1}{2}\om_n\not\in\Z\Phi(G)$
 \ii $\frac{n(p-1)}{2}$ is odd. Observe that $\om_n-\om_{n-1}\not\in\Z\Phi(G)$ 
and $\om_{n-1}+\frac{p-3}{2}\om_n=\om_{n-1}-\om_n+\frac{p-1}{2}\om_n$; it follows that 
$\om_{n-1}+\frac{p-3}{2}\om_n \not\in\Z\Phi(G)$ \ii $\frac{n(p-1)}{2}$ is even, as stated in Table~\ref{tab:Omega1-symp}.

Let $G= D_n$, $n>3$ odd; here the weights which we must consider are $\om_1$, $2\om_1$ and $\om_2$. 
Since $L_G(\om_1)$ is the natural orthogonal module for $G$, $h$ acts trivially on $L_G(\om_1)$. 
Since $2\om_1$ and $\om_2$ each occur in $L_G(\om_1)\otimes L_G(\om_1)$, $L_G(2\om_1)$ and $L_G(\om_2)$ 
are also orthogonal.

Let $G= D_n$, $n>3$ even. For the weights $\om_1$, $2\om_1$ and $\om_2$, the argument of the 
previous paragraph
is valid. So we must consider the weights $\om_{n-1}$ and $\om_n$. When $n=4$, a direct check using the information in \cite[Planche IV]{Bourb4-6} allows one to see that $\om_{n-1}(h) = (-1)^{\frac{n(n-1)}{2}} = \om_{n}(h)$,
and so $L_G(\om_{n-1})$ and $L_G(\om_n)$ are  symplectic if and only if $n\equiv 2\pmod 4$.

Let $G= B_n$, $n>2$. In this case $\om_n$ is the only fundamental dominant weight
which does not lie in the root lattice. In the natural embedding of $B_n$ in $D_{n+1}$, $B_n$ acts
irreducibly on each of the spin modules for $D_{n+1}$;
the restriction is the $FG$-module $L_{B_n}(\om_n)$ in each case. Hence,
for $n$ odd, $L_G(\om_n)$ is symplectic if and only if $n+1\equiv2\pmod 4$. When $n$ is even,
we consider the natural embedding of $D_n$ in $B_n$, where the spin module for $B_n$ decomposes
as a direct sum of the two distinct spin modules for $D_n$, and the summands are non-degenerate
with respect to the form. Hence $L_G(\om_n)$ is symplectic if and only if 
$n\equiv 2\pmod 4$. So to summarize, for the group $G=B_n$, $L_G(\om_n)$ is symplectic if and only if $n\equiv 1$ or $2\pmod 4$.

This completes the proof of the proposition.\end{proof}

\begin{table}
$$\begin{array}{|l|c|c|c|}
\hline
G&\lam&\mbox{conditions}&\dim L_G(\lam)\\
\hline
A_1&a\om_1&a~\mbox { odd }&a+1\\
\hline
&&&\\
A_n&\om_{(n+1)/2}&n>1 \mbox{ odd},\frac{n+1}{2}\mbox{ odd}&\begin{pmatrix}n+1\cr \frac{n+1}{2}\end{pmatrix} 
\\
&&&\\
\hline
B_n&\om_n&n>2, n\equiv 1\mbox{ or } 2\pmod 4&2^n\\
\hline
C_3&\om_3&&14\\
C_n&\om_1& n>1&2n\\
&&&\\
&\om_{n-1}+\frac{p-3}{2}\om_n& n>1, p\geq 3,n(p-1)/2 \mbox{ even}&(p^n-1)/2\\
&&&\\
&\frac{p-1}{2}\om_n& n>1, p\geq 3,n(p-1)/2~\mbox{ odd }&(p^n+1)/2\\
&&&\\

\hline
D_n&\om_{n-1},~ \om_{n}& n>3~\mbox{ even}, n\equiv 2\pmod 4&2^{n-1}\\
\hline
E_7&\om_7&&56\\
\hline
\end{array}$$
\caption{$\lam\in\Omega_2(G)$, $p\ne2$, $L_G(\lam)$\mbox{ symplectic }}\label{tab:Omega1-symp}
\end{table}

Continuing with the case where $G$ preserves a non-degenerate form on $L_G(\lambda)$, for the weights $\lam\in \Omega_2(G)$ not listed in
Table~\ref{tab:Omega1-symp}, the module $L_G(\lam)$ is orthogonal. In order
to decide whether the image of $G$ under the corresponding representation belongs 
to a subgroup of type $B_n$ or $D_n$,
one has only to determine whether $\dim L_G(\lam)$ is odd or even. Since we are interested in 
the solution
to Problem 2, we consider those weights $\lam\in\Omega_2(G)$ for which the multiplicity of the 
0 weight
is at most 1, if $L_G(\lam)$ is odd-dimensional, and at most 2 if $\dim L_G(\lam)$ is even.

The following lemma can be deduced directly from \cite{Lubeck} and the preceding results.
\begin{lemma}\label{ed4} Let $\lambda\in\Omega_2(G)$. Assume $p>2$, $L_G(\lam)$ is orthogonal,
and moreover the multiplicity of
the $0$ weight in $L_G(\lambda)$ is at most $2$.
Then $\dim L_G(\lam)$ is even if and only if the pair $(G,\lam)$ is as in Table~{\rm \ref{tab:omega2-even-orth}}.
\end{lemma}

\begin{table}
$$\begin{array}{|l|c|c|c|}
\hline
G& \lam&\mbox{conditions}&\dim L_G(\lam)\\
\hline
&&&\\
A_n&\om_{(n+1)/2}& n>1 \mbox{ odd}, (n+1)/2~\mbox{ even}&\begin{pmatrix}n+1\cr \frac{n+1}{2}\end{pmatrix}\\
&&&\\
A_2&\om_1+\om_2& p\ne3&8\\
A_3&2\om_2& p> 3&20\\
\hline
B_n&\omega_n& n\equiv 0\mbox{ or } 3\pmod4&2^n\\
\hline
C_2&2\om_1&&10\\
&2\om_2&p\ne 5&14\\
C_3&\om_2& p>3&14\\
C_4&\omega_4& p>3&42\\
\hline
D_n&\om_1& n>3&2n\\
&\om_{n-1},\om_n&n\equiv 0\pmod4&2^{n-1}\\
\hline
F_4&\om_4&p>3&26\\
\hline
G_2&\om_2&p>3&14\\
\hline
\end{array}$$
\caption{$\lambda\in\Omega_2(G)$, $p\ne2$, $L_G(\lam)$\mbox{ even-dimensional orthogonal, with} 0 \mbox{weight of multiplicity at most} 2}\label{tab:omega2-even-orth}
\end{table}

We give the odd-dimensional orthogonal representations $L_G(\lam)$, with 
$\lam\in\Omega_1(G)$ in Table~\ref{tab:Omega1-odd-orth}.

\begin{table}
$$\begin{array}{|l|c|c|c|}
\hline
G& \lam&\mbox{conditions}&\dim L_G(\lam)\\
\hline
A_1&a\om_1&a>0\mbox{ even}&a+1\\
\hline 
A_n&(p-1)\om_{(n+1)/2}& n>1,n\mbox{ odd}&{\rm not ~known}\\
&\frac{p-1}{2}(\om_{n/2}+\om_{(n+2)/2})&n\mbox{ even}&{\rm not~ known}\\
\hline
B_n&\om_1& n>2&2n+1\\
\hline
C_2&\om_2&&5\\
\hline
C_n&\om_{n-1}+\frac{p-3}{2}\om_n& n\geq2, \frac{n(p-1)}{2}\mbox{ odd}&(p^n-1)/2\\
&&&\\
&\frac{p-1}{2}\om_n& n\geq2, \frac{n(p-1)}{2},\mbox{ even}&(p^n+1)/2\\
\hline
F_4&\om_4&p=3&25\\
\hline
G_2&\om_1&&7\\
&\om_2&p=3&7\\
\hline
\end{array}$$
\caption{$\lam\in\Omega_1(G)$, $p\ne2$, $L_G(\lam)$\mbox{ odd-dimensional orthogonal }}\label{tab:Omega1-odd-orth}
\end{table}

\begin{lemma}\label{e24} Let $\lambda\in\Omega_2(G)$ such that  
$L_G(\lambda)$ is self-dual and
 the multiplicity of the $0$ weight in $L_G(\lambda)$ is  at
most $2$. Assume in addition that $p=2$. Then $L_G(\lambda)$ has
a non-degenerate $G$-invariant quadratic form if and only if $(G, \lam)$ are as in 
Table{\rm ~\ref{tab:orthp=2}}.
\end{lemma}
\begin{table}
$$\begin{array}{|l|c|c|c|}
\hline
G&\lam&\mbox{conditions}&\dim L_G(\lam)\\
\hline
&&&\\
A_n&\om_{(n+1)/2}& n>1\mbox{ odd}&\begin{pmatrix}n+1\cr \frac{n+1}{2}\end{pmatrix}\\
&&&\\
\hline
A_2&\om_1+\om_2&&8\\
\hline
A_3&\om_1+\om_3&&14\\
\hline
C_n&\om_n&&2^n\\
\hline
C_3&\om_2&&14\\
\hline
C_4&\om_2&&26\\
\hline
D_n&\om_1&n>3&2n\\
\hline
D_n&\om_{n-1},\om_{n}& n>3\mbox{ even}&2^{n-1}\\
\hline
D_4&\om_2&&26\\
\hline
E_7&\om_7&&56\\
\hline
F_4&\om_1,\om_4&&26\\
\hline
G_2&\om_2&&14\\
\hline
\end{array}$$
\caption{$\lambda\in\Omega_2(G)$, $p=2$, $L_G(\lam)$ \mbox{orthogonal, with }0\mbox{ weight  
of multiplicity} \mbox{ at most } 2}\label{tab:orthp=2}
\end{table}

\begin{proof} We first inspect the last column of Table~\ref{tab:omega2}, where the dimension
of the 0 weight space in $L_G(\lambda)$ is given. In addition, the result of
Proposition~\ref{w2t}, and the remarks preceding the
proposition, further restrict the list of pairs $(G,\lambda)$ which must be
considered. We find that the only even-dimensional
$L_G(\lambda)$ in addition to those listed in Table~\ref{tab:orthp=2}
are as follows:
\begin{enumerate}[a)]
\item $G =A_1$, $\lambda = \omega_1$, $\dim L_G(\lam) = 2$;

\item $G=C_n$, $\lambda = \omega_1$, $\dim L_G(\lambda) =2n$;

\item $G=G_2$, $\lambda = \omega_1$, $\dim L_G(\lambda) = 6$.

\end{enumerate}

In cases (a), (b) and (c) above, it is well-known that
 $G$ preserves no non-degenerate
quadratic form on $L_G(\lambda)$. Hence we now turn to the list
of pairs $(G,\lambda)$ of Table~\ref{tab:orthp=2}. The pair $(D_n,\om_1)$
is clear as $L_{D_n}(\omega_1)$ is the natural representation of the classical group 
of type $D_n$.

For cases where $L_G(\lambda)$ occurs as a composition factor of the adjoint
representation of $G$, we refer to \cite[\S3]{GW1} to conclude that
$G$ preserves a non-degenerate
quadratic form on $L_G(\lambda)$. (Note that we may also apply the exceptional
isogeny to the group $F_4$.) The pairs $(C_3,\om_2)$, $(C_4,\om_2)$, $(E_7,\om_7)$ are covered 
in 
\cite[Table 2]{Br1}. This leaves us with the pairs $(A_n,\om_{(n+1)/2})$, 
$(C_n,\om_n)$, $(D_n,\om_j)$, $j=n-1,n$.
For the second case, we may consider the group $B_n$ acting on $L_G(\om_n)$. Then in
all cases the Weyl module with the given highest weight is irreducible and the
existence of a $G$-invariant quadratic form follows from \cite[2.4]{SW}.\end{proof}

Additionally, we provide in Table~\ref{tab:nonselfdual} the 
list of modules $L_G(\lam)$,
with $\lam\in\Omega_1(G)$, $L_G(\lam)\not\cong L_G(\lam)^*$. So in particular, the image of $G$ under the corresponding
representation contains a regular torus of $H = {\rm SL}(L_G(\lam))$ and $G$ does not lie in  a proper 
classical subgroup of $H$. Finally, Table~\ref{tab:p=2symp} records the 
non-orthogonal symplectic modules $L_G(\lam)$ for $\lam\in\Omega_1(G)$ when $p=2$.

\begin{table}
$$\begin{array}{|l|c|c|}
\hline
G&\lam&\mbox{conditions}\\
\hline
A_n&a\om_1, b\om_n& n>1, 1\leq a,b <p\\
&&\\
&\om_i& 1<i<n, i\ne(n+1)/2\mbox{ if } n\mbox{ is odd}\\
&&\\
&c\om_i+(p-1-c)\om_{i+1}& 1\leq i<n, 0\leq c\leq p-1, \mbox{ and }\\
&&c\ne (p-1)/2\mbox{ if } n\mbox{ is even and } i=n/2;\\
&&c\ne0\mbox{ if } n\mbox{ is odd and } i=(n-1)/2;\\
&&c\ne p-1\mbox{ if } n\mbox{ is odd and } i=(n+1)/2.\\
\hline
D_n&\om_{n-1},\om_{n}& n>3\mbox{ odd }\\
\hline
E_6&\om_1,\om_6&\\
\hline
\end{array}$$
\caption{$\lam\in\Omega_1(G)$, $\lam\ne-w_0\lam$\mbox{ and $\phi(G)$ contains a regular torus of} $\SL(L_G(\lambda))$}\label{tab:nonselfdual} 
\end{table}

\bigskip

\begin{table}
$$\begin{array}{|l|c|c|c|}
\hline
G&\lam&\mbox{conditions}&\dim L_G(\lam)\\
\hline
C_n &\om_1&n\geq1&2n\\
\hline
G_2&\om_1&&6\\
\hline
\end{array}$$
\caption{$p=2$, $\lam\in\Omega_1(G)$, $L_G(\lam)$ non-orthogonal symplectic}\label{tab:p=2symp}
\end{table}

We can now give an explicit solution to Problem 1 for simple subgroups $G$ of 
classical groups. In
Proposition~\ref{ma1}, we treat the case of tensor-indecomposable irreducible 
representations of $G$ 
all of whose weight spaces are 1-dimensional. As discussed in Section~\ref{sec:red}, the image of 
$G$ under a tensor-decomposable
representation is not maximal in the classical group. In Proposition~\ref{ma2}, we 
handle the orthogonal
irreducible representations of $G$ whose zero weight space has dimension 2 while 
all other weight spaces are 1-dimensional.
In the following two Propositions, we determine 
whether the image of $G$ under 
the given representation is a 
maximal subgroup
of the minimal classical group containing it.

\begin{propo}\label{ma1} Let $G$ be a simple algebraic group and let 
$\lam\in\Omega_1(G)$. Let $\rho$ be 
an irreducible \rep of $G$ with highest weight $\lam$.
Then $\rho(G)$ is a maximal subgroup in the minimal classical group containing 
$\rho(G)$, except for the following cases:

$(1)$ $G=A_1$, $\lam=6\om_1$, $p\geq 7$, where there is an intermediate subgroup 
of type $G_2;$

$(2)$ $G=A_2$, $\lam=\om_1+\om_2$, $p=3$, where there is an intermediate subgroup 
of type $G_2;$

$(3) $ $G=B_n$, $\lam=\om_n$, $p\neq 2$, where  there is an intermediate subgroup 
of type  $D_{n+1};$ 

$(4) $ $G=C_n$, $\lam=\om_n$, $p= 2$, where  there is an intermediate subgroup 
of type  $D_{n+1}.$ 
\end{propo}

\begin{proof} This follows from  Seitz's classification of maximal closed
connected subgroups of the classical type simple algebraic groups, see \cite[Thm. 3, Table 1]{Seitzclass} 
and our results above.\end{proof}

\begin{propo}\label{ma2} Let $\lam$ be a weight of $G$ occurring in Tables $\ref{tab:omega2-even-orth}$ or $\ref{tab:orthp=2}$ but not in Table $\ref{tab:omega1}$.
  Let $\rho$ be an irreducible
\rep with highest weight $\lam$. In particular, $\rho(G)$ lies in a classical
group of type $D_m$, and does not contain a regular torus of ${\rm SL}(L_G(\lambda))$.
Then $\rho(G)$ is a maximal subgroup of the orthogonal group containing $\rho(G)$, except for the following cases:

$(1)$ $G=A_3$, $\lam=\om_1+\om_3$, $p=2 $, where there is an intermediate subgroup 
of type $C_3$;

$(2)$ $G=D_4$, $\lam=\om_2$, $p=2 $, where there are intermediate subgroups 
of type $C_4$ and $F_4$;
 
$(3)$ $G=G_2$, $\lam=\om_2$, $p=2 $, where there is an intermediate subgroup 
of type $C_3.$

\end{propo}

\begin{proof} The proof is carried out exactly as the proof of Proposition~\ref{ma1}.
\end{proof}

\section{Maximal reductive subgroups of exceptional groups  containing  regular tori }\label{sec:exc}

In this section, we consider Problem 1 for the case where $H$ is an exceptional type 
simple algebraic group
over $F$. 
We will determine
all maximal closed connected subgroups $M$ of $H$ which contain a regular
torus. As discussed in Section~\ref{sec:red}, we will assume $M$ to be 
reductive and ${\rm rank}(M)<{\rm rank}(H)$.  The main tool
is the classification of the maximal closed connected subgroups of $H$, as given in
\cite[Cor. 2(ii)]{LS2}.

For a semisimple group $M = M_1M_2\cdots M_t$, with $M_i$ simple,
and with respect to a fixed maximal torus $T_M$ of $M$, we will
write $\{\omega_{i1},\dots,\omega_{i\ell_i}\}$, for the set of
fundamental dominant weights of $T_M\cap M_i$ (so ${\rm rank}(M_i)
= \ell_i$). In case $M$ is simple, we will simply write
$\{\omega_1,\dots,\omega_\ell\}$.

\begin{propo}\label{dd1} Let $M$ be a maximal closed connected positive-dimensional
subgroup of an exceptional algebraic group $H$. Assume
${\rm rank}(M)<{\rm rank}(H)$. Then $M$ contains a regular torus
of $H$ if and only if the pair $(M,H)$ is as given in Table~{\rm \ref{tab:exc}}. 
In particular, if $M$ contains a regular torus of $H$, then $M$ is semisimple.
\end{propo}

Before proving the result, it is interesting to compare the above table with \cite[Thm. 1]{LS2},
which  describes the maximal closed connected positive-dimensional subgroups of the exceptional simple algebraic groups.
There are precisely four pairs ($M,H)$, $M$ a maximal closed connected 
positive-dimensional subgroup of
an exceptional algebraic group $H$ with ${\rm rank}(M)<{\rm rank}(H)$, and where $M$ does 
\emph{not} contain a regular torus of $H$:
 one class of $A_1$ subgroups in $H=E_7$, 2 classes of $A_1$ subgroups in $H=E_8$ and a maximal $B_2$ in $E_8$.

\begin{table}
$$\begin{array}{|l|l|l|}\hline
H&M\mbox{simple}&M\mbox{non simple}\\\hline\hline
G_2&A_1\  (p\geq 7)&\\
\hline
F_4& A_1\ (p\geq 13), G_2\ (p=7)&A_1G_2\ (p\geq3)\\
\hline
E_6&A_2\ (p\geq5), G_2\ (p\ne 7)&A_2G_2\\
&C_4\ (p\geq3), F_4&\\
\hline
E_7&A_1\ (p\ge 19), A_2\ (p\ge 5)&A_1A_1\ (p\geq5), A_1G_2\ (p\geq3)\\
&&A_1F_4, G_2C_3\\
\hline
E_8&A_1\ (p\geq31)&A_1A_2\ (p\geq5), G_2F_4\\
\hline
\end{array}$$
\caption{Maximal connected reductive subgroups $M\subset H$, $H$ exceptional, ${\rm rank}(M)<{\rm rank}(H)$,
with $M$ containing a regular torus of $H$}\label{tab:exc}
\end{table}

\begin{proof} Let $M$ be as in the statement of the result and fix $T_M$, a maximal torus of
$M$. (Throughout the proof we will refer to $M$ as a \emph{maximal} subgroup, even 
though $M$ may only be maximal among connected subgroups.) Then \cite[Cor. 2]{LS2} implies that $M$ is semisimple.
The method of proof is quite simple. By Proposition~\ref{lie},
$T_M$ is a regular torus in $H$ if and only if 
$\dim(C_{{\rm Lie}(H)}(T_M)) = {\rm rank}(H)$. Hence, we need only determine the dimension of the
$0$ weight space for $T_M$ acting on ${\rm Lie}(H)$. This can be
deduced from the information in \cite[Table 10.1]{LS2}.

If $M$ is of type $A_1$, the notation $T(m_1;m_2;\dots;m_k)$, used
in \cite[Table 10.1]{LS2}, represents an $FM$-module whose
composition factors are the same as those of
$W_M(m_1\omega_1)\oplus\cdots\oplus W_M(m_k\omega_1)$. Since the
multiplicity of the $0$ weight in each Weyl module for $A_1$ is
precisely 1, we see that the only maximal $A_1$-subgroups
containing a regular torus are those listed above. This covers the
case $H=G_2$.

Consider now the two remaining cases in $H=F_4$. If $M$ is the
maximal $G_2$ subgroup in $H$ (occurring only for $p=7$), then
${\rm Lie}(H)|_M$ has composition factors $L_M(\omega_2)$ and
$L_M(\omega_1+\omega_2)$. Now consulting \cite{luebeck}, we see that
the $0$ weight has multiplicity 2 in each of these irreducible
modules and hence multiplicity 4 in ${\rm Lie}(H)$. This then
implies that $T_M$ is a regular torus in $H$. For the semisimple
subgroup $M = A_1G_2$ in $F_4$ (which exists when $p\geq3$),
 we must explain an additional notation used in \cite{LS2}.
In \cite[Table 10.1]{LS2}, the notation $\Delta(\mu_1;\mu_2)$
denotes a certain indecomposable $FM$-module whose composition factors
are $L_M(\mu_1)$, $L_M(\mu_2)$ and two factors $L_M(\nu)$, where $\mu_1$ and
$\mu_2$ are dominant weights such that the tilting modules
$T(\mu_1)$ and $T(\mu_2)$ each have socle and irreducible quotient
of highest weight $\nu$. Now if $p> 3$, ${\rm Lie}(H)|_M$ has
composition factors $L_M(4\omega_{11}+\omega_{21})$,
$L_M(2\omega_{11})$, and $L_M(\omega_{22})$; the multiplicity of the
$0$ weight in these modules is 1, 1, 2, respectively, and hence
$T_M$ is a regular torus of $H$. In case $p=3$, the composition
factors are $L_M(4\omega_{11}+\omega_{21})$, $L_M(\omega_{22})$
$L_M(\omega_{21})$, $L_M(\omega_{21})$, and $L_M(2\omega_{11})$, and
again the multiplicity of the
 0 weight is 4. This completes the consideration of the case $H=F_4$.

Now we consider the case $H = E_6$ and ${\rm rank}(M)\geq2$. There
is a
 maximal $A_2$ subgroup $M$ of $H$ (when $p\geq 5$) whose action on
${\rm Lie}(H)$ is ${\rm Lie}(H)|_M = L_M(4\omega_1+\omega_2)\oplus
L_M(\omega_1+4\omega_2)\oplus L_M(\omega_1+\omega_2)$. Consulting
\cite{luebeck}, we see that the multiplicity of the zero weight in
this module is 6 and so $M$ contains a regular torus of $H$. The
group $H$ also has a maximal $G_2$ subgroup $M$ when $p\ne 7$, such
that ${\rm Lie}(H)|_M$ has the same set of composition factors as
$W_M(\omega_1+\omega_2)\oplus W_M(\omega_2)$. Now we consult
\cite{luebeck} and find that the multiplicity of the zero weight
in ${\rm Lie}(H)|_M$ is $6$ for all characteristics $p\ne7$; hence
$M$ contains a regular torus of $H$.

Turn now to the maximal closed connected subgroups of $H=E_6$, of rank at
least 4. The maximal subgroup $M$ of type $C_4$ acts on
${\rm Lie}(H)$ with composition factors $L_M(2\omega_1)$ and
$L_M(\omega_4)$, if $p\ne 3$, and with these same composition factors
plus an additional 1-dimensional composition factor, if $p=3$. As
usual, we find that the dimension of the 0 weight space is 6 and
so $C_4$ contains a regular torus of $H$. For the maximal $F_4$
subgroup $M$ of $H$, which exists in all characteristics, ${\rm
Lie}(H)|_M = L_M(\omega_4)\oplus L_M(\omega_1)$, if $p>2$, and when
$p=2$, the Lie algebra is isomorphic to the tilting module of
highest weight $\omega_1$, which has a composition factor
$L_M(\omega_1)$ and two factors $L_M(\omega_4)$. The usual argument shows
that $M$ contains a regular torus of $H$. Finally, we consider the
maximal subgroup $M\subset H$ of type $A_2G_2$; ${\rm
Lie}(H)|_M$ has the same set of $T_M$-weights (and multiplicities) as the $FM$-module
$W_M(\omega_{11}+\omega_{12}+\omega_{21})\oplus
W_M(\omega_{11}+\omega_{12})\oplus W_M(\omega_{22})$. One checks as
usual that the multiplicity of the 0 weight is indeed 6.

We now turn to the case $H = E_7$, and $M$ is a maximal closed connected
subgroup of rank $2$. There exists a maximal $A_2$ subgroup $M$ of $H$
when $p\geq 5$, whose action on ${\rm Lie}(H)$ is given by
$L_M(4\omega_1+4\omega_2)\oplus L_M(\omega_1+\omega_2)$, when $p\ne
7$, and ${\rm Lie}(H)|_M = T(4\omega_1+4\omega_2)$, when $p=7$.
Again using \cite{luebeck} one verifies the multiplicity of the 0 weight
 in ${\rm Lie}(H)|_M$ is 7 and hence $T_M$ is a regular torus
of $H$. The maximal $A_1A_1$ subgroup $M$ of $H$ is such that
${\rm Lie}(H)|_M$ has the same set of $T_M$-weights (and multiplicities) as the module
$W_M(2\omega_{11}+8\omega_{21})\oplus
 W_M(4\omega_{11}+6\omega_{21})\oplus  W_M(6\omega_{11}+4\omega_{21})\oplus  
W_M(2\omega_{11}+4\omega_{21})\oplus
 W_M(4\omega_{11}+2\omega_{21})\oplus W_M(2\omega_{11}) \oplus  W_M(2\omega_{21})$.
One verifies that the multiplicity of the zero weight is $7$ and
hence $M$ contains a regular torus of $H$. This completes the
consideration of the rank two reductive maximal connected subgroups.

We now handle the remaining maximal connected subgroups of $H=E_7$. The maximal
$A_1G_2$ subgroup $M$ of $H$, which exists for all
$p\geq 3$, satisfies: ${\rm Lie}(H)|_M$ has the same set of
$T_M$-weights (and multiplicities) as the $FM$-module $W_M(4\omega_{11}+\omega_{21})\oplus
W_M(2\omega_{11}+2\omega_{21})\oplus W_M(2\omega_{11}) \oplus
W_M(\omega_{22})$. As usual, we check that the multiplicity of the
zero weight in ${\rm Lie}(H)|_M$ is 7. The maximal
$A_1F_4$ subgroup $M$ is such that ${\rm Lie}(H)|_M$ has
the same set of $T_M$-weights (and multiplicities) as the module
$W_M(2\omega_{11}+\omega_{24})\oplus W_M(2\omega_{11})\oplus
W_M(\omega_{21})$. This module has a 7-dimensional 0 weight space
and so $M$ contains a regular torus of $H$. Finally, we consider
the maximal $G_2C_3$ subgroup $M$, which
exists in all characteristics. In this case, ${\rm Lie}(H)|_M$ has
the same set of $T_M$-weights (and multiplicities) as the $FM$-module
$W_M(\omega_{11}+\omega_{22})\oplus W_M(\omega_{12})\oplus
W_M(2\omega_{21})$, which has a 7-dimensional $0$ weight space and
again $M$ contains a regular torus of $H$.

To complete the proof, we now turn to the case $H=E_8$ and $M$ a
maximal closed connected subgroup of rank at least 2. There exists a
unique (up to conjugacy) rank $2$ reductive maximal subgroup of $H$, namely
$M = B_2$, when $p\geq 5$. Here ${\rm Lie}(H)|_M$ has the same set
of $T_M$-weights (and multiplicities) as $W_M(6\omega_2)\oplus
W_M(3\omega_1+2\omega_2)\oplus W_M(2\omega_2)$; but this latter has a
12-dimensional 0 weight space and so $T_M$ is not a regular torus
of $H$. We now consider the group $M = A_1A_2$. Here ${\rm
Lie}(H)|_M$ has the same set of $T_M$-weights (and multiplicities) as the $FM$-module
$W_M(6\omega_{11}+\omega_{21}+\omega_{22})\oplus
W_M(2\omega_{11}+2\omega_{21}+2\omega_{22})\oplus
W_M(4\omega_{11}+3\omega_{21})\oplus
W_M(4\omega_{11}+3\omega_{22})\oplus W_M(2\omega_{11})\oplus
W_M(\omega_{21}+\omega_{22})$, which has an 8-dimensional 0 weight
space and hence $M$ contains a regular torus of $H$. Finally, we
consider the maximal $G_2F_4$ subgroup $M$, which
exists in all characteristics. Here ${\rm Lie}(H)|_M$ has the same
set of $T_M$-weights (and multiplicities) as the $FM$-module
$W_M(\omega_{11}+\omega_{24})\oplus W_M(\omega_{12})\oplus
W_M(\omega_{21})$, which has an 8-dimensional 0 weight space and so
$T_M$ is a regular torus of $H$.
\end{proof}


\section{Non-semisimple regular elements}\label{sec:nonss}

In this section, we prove Theorem~\ref{mt33}, and therefore reduce Problem 1 to the 
consideration of semisimple elements, the case which has been discussed in detail in the 
preceding sections. Recall that throughout the paper we are assuming $p\ne 2$ when $G$ is a simple group of type $B_n$.

The proof of the following Lemma was provided by Iulian Simion. Our original proof was based upon the classification result of 
Theorem~\ref{tt2}, and required a case-by-case analysis. We are very thankful to I. Simion for suggesting this more elegant proof.

\begin{propo}\label{cc1} Let $G$ be a closed connected reductive subgroup of a reductive algebraic group $H$.
Suppose that 
$G$ contains a regular unipotent element of $H$.
 Then every maximal torus in $G$ is a regular torus in $H$.\end{propo}

\begin{proof} Let $T_G$ be a maximal torus of $G$. Then $C_H(T_G)$ is a Levi subgroup of $H$. (See for example \cite[Prop.12.10]{MTbook}.)
Let $B_H = T_HR_u(B_H)$ be a Borel subgroup of $H$ such that 
$P=\langle B_H,C_H(T_G)\rangle$ is a parabolic subgroup of $H$
and $T_G\subset T_H$.
In this proof, we will write $g\cdot x$ to denote the action of $g$ on $x$ by conjugation, that is $g\cdot x = gxg^{-1}$.
 By assumption, $G$ contains a regular unipotent element of $H$ and \cite[Lemma 2.1]{TeZ} then implies that every 
regular unipotent element of $G$ is regular in $H$.
Semisimple elements are dense in $G$ and hence ${\overline{G\cdot T_G}}$ contains a regular unipotent element
of $G$, and hence a regular unipotent element of $H$.
Now since $H/B_H$ is complete, \cite[2.13, Lemma 2]{St_cong} implies that for any closed $B_H$-invariant subset $X\subset H$, 
$H\cdot X$ is closed. 
So in particular, since $B_H\cdot T_G$ is $B_H$-stable, so is the closed set $\overline{B_H\cdot T_G}$ and so 
$H\cdot\overline{B_H\cdot T_G}$ is closed and we have $$\overline{G\cdot T_G}\subset \overline{H\cdot T_G}=\overline{H\cdot (B_H\cdot T_G)} \subset H\cdot\overline{B_H\cdot T_G}.$$ On the other hand, $\overline{B_H\cdot T_G}$ lies in the radical of $P$, and so any 
unipotent class intersecting 
$H\cdot\overline{B_H\cdot T_G}$ is represented in $R_u(P)$.
Since $\overline{G\cdot T_G}$ contains regular unipotent elements of $H$, we see that $P=B_H$ and so $C_H(T_G) = T_H$
and $T_G$ is a regular torus of $H$.\end{proof}

\begin{proof}[Proof of Theorem~$\ref{mt33}$]  Let $g\in G$ be a regular element of $H$. Then $g=su$ where $s\in G$ is semisimple, 
$u\in G$ is unipotent and 
$su=us$. Set $Y=C_H(s)^\circ$ and $X=C_G(s)^\circ$. We first claim that $u$ is a regular unipotent element of $Y$. (See \cite[Cor. 4.4]{SpSt}.)
Indeed, if $u$ is not regular in $Y$ then $\dim C_H(g) = \dim C_Y(u)> {\rm rank}(Y)={\rm rank}(H)$, contradicting the regularity of $g$. 

By Proposition~\ref{cc1} applied
 to the connected reductive groups $X\subset Y$, a maximal torus $T_X$ of
$X$ is regular in $Y$. Taking $T_X$ such that $s\in T_X$, we have $C_H(T_X)\subset C_H(s)^\circ = Y$
and so $C_H(T_X)\subset C_Y(T_X)$, which is a torus. Now apply Lemma~\ref{ca1} to obtain the reuslt.\end{proof}

We conclude the paper with some remarks about how one can determine the reductive overgroups
of a general regular element, neither semisimple no unipotent. Let $su =g\in G\subset H$ be as in the above proof.
 Then $u$ is a regular unipotent element
in the connected reductive group $Y = C_H(s)^\circ$, lying in the connected reductive group $X=C_G(s)^\circ$.
We now appeal to the following classification:

\begin{theo}{\rm \cite[Theorem 1.4]{TeZ}}\label{tt2} Let $G$ be a closed semisimple 
subgroup of the 
simple algebraic group $H$,
containing a regular unipotent element of $H$. Then $G$ is simple and either the
pair $(H,G)$ is as given in Table~$\ref{tab:pairs}$, or $G$
is of type $A_1$ and $p=0$ or $p\geq h$, where $h$ is the
Coxeter number for $H$. Moreover, for each pair of root systems $(\Phi(H),\Phi(G))$ as in 
the table,
respectively, for $(\Phi(H),A_1,p)$, with $p=0$ or $p\geq h$, there exists a closed simple 
subgroup $X\subset H$ of 
type 
$\Phi(G)$, respectively $A_1$, containing
a regular unipotent element of $H$.\end{theo}

\begin{table}
\[\begin{array}{|l|l|}
\hline
H&G\cr
\hline
A_6&G_2, p\ne2\cr
A_5&G_2,p=2\cr
\hline
C_3&G_2, p=2\cr
\hline
B_3&G_2, p\ne2\cr
\hline
D_4&G_2, p\ne 2\cr
&B_3 \cr
\hline
E_6&F_4\cr
\hline
A_{n-1}, n>1&C_{n/2}, n\mbox{ {\rm even}}\cr
&B_{(n-1)/2}, n\mbox{ {\rm odd}}, p\ne2\cr
\hline
D_n, n>4&B_{n-1}\cr
\hline
\end{array}\]
\caption{Semisimple subgroups $G\subset H$ containing a regular unipotent element} \label{tab:pairs}
\end{table}

The above classification applies to simple groups $H$, but it is straightforward to deduce from this the set of 
possible pairs of reductive groups $(C_G(s)^\circ,C_H(s)^\circ)$. Coupling this with information about centralizers of   
semisimple elements in $H$ and in $G$, we can obtain a classification of the reductive subgroups of $H$ containing 
$g$.

\begin{remar} We point out here an inaccuracy in \cite{SS} and \cite{TeZ}. In the statement of \cite[Thm.~A]{SS},
the condition given for the existence of an $A_1$ subgroup containing a regular unipotent element is $p>h$. But in fact,
as shown in \cite{TestA1}, such a subgroup exists for all $p\geq h$. This has been correctly stated in Theorem~$\ref{tt2}$.
\end{remar}

\bibliographystyle{alpha}
\bibliography{TZ_arxiv}
\end{document}